\documentclass[11pt]{article}

\usepackage{setspace,graphicx,amssymb,amsmath,mathtools,latexsym,amsfonts,amscd,amsthm,multirow,mathdots,caption,array,dsfont,wrapfig}
\usepackage{fancyhdr,tabularx,mathrsfs,enumitem,graphicx}
\usepackage{stmaryrd}
\usepackage{rotating}
\usepackage{authblk}
\usepackage{hyperref}
\usepackage{comment,enumitem}
\usepackage{geometry}
\usepackage{arydshln}
\usepackage{longtable}
\usepackage{hyperref}
\usepackage{booktabs}
\usepackage{braket}
\usepackage{subcaption} 
\usepackage{lineno}
\usepackage{verbatim}
\usepackage[toc,page]{appendix} 
\usepackage[backend=biber, sorting=none, citestyle=numeric]{biblatex}
\usepackage[font=small,labelfont=bf]{caption}
\usepackage{tikz}
\usetikzlibrary{calc,trees,positioning,arrows,chains,shapes.geometric,%
	decorations.pathreplacing,decorations.pathmorphing,shapes,%
	matrix,shapes.symbols,decorations.markings}
\usepackage{float}
\newcommand{\bea}{\begin{eqnarray}} 
\newcommand{\eea}{\end{eqnarray}} 
\newcommand{\bee}{\begin{eqnarray*}} 
	\newcommand{\eee}{\end{eqnarray*}} 
\newcommand{\al}{\begin{align*}} 
\newcommand{\eal}{\end{align*}} 
\newcommand{\be}{\begin{equation}} 
\newcommand{\ee}{\end{equation}} 
 
\newcommand{\bem}{\begin{pmatrix}} 
	\newcommand{\eem}{\end{pmatrix}} 

\newcommand{\GS}[1]{{\color{green}{GS: #1}}}

\newcolumntype{R}{ >{$}r <{$}}
\newcolumntype{C}{ >{$}c <{$}}
\newcolumntype{L}{ >{$}l <{$}}
\newcolumntype{F}{>{\centering\arraybackslash}m{1.5cm}}

\newcommand{\ZZ}{{\mathbb Z}}%Integers
\newcommand{\QQ}{{\mathbb Q}}%Rationals
\newcommand*{\rA}{\text{\rotatebox[origin=c]{180}{$A$}}}
\newtheorem{thm}{Theorem}[section]
\newtheorem{cor}[thm]{Corollary}
\newtheorem{lem}[thm]{Lemma}
\newtheorem{prop}[thm]{Proposition}

\theoremstyle{definition}

\theoremstyle{remark}
\newtheorem{rmk}[thm]{Remark}
\AfterEndEnvironment{rmk}{\noindent\ignorespaces}

\numberwithin{equation}{section}
\title{
	\vspace{-35pt}
	\textsc{\Large{
			Cone Vertex Algebras, Mock theta functions, \\and Umbral Moonshine modules
	}  }
}

\author[1]{Miranda C. N. Cheng \thanks{Corresponding author.}}
\author[2]{Gabriele Sgroi \thanks{First author.}}
\affil[1]{Korteweg-de Vries Institute for Mathematics, University of Amsterdam, the Netherlands}
\affil[1,2]{Institute of Physics, University of Amsterdam,  the Netherlands}
\affil[1]{Institute for Mathematics, Academica Sinica, Taipei, Taiwan}

\begin{document}
\date{}
\maketitle

	\setstretch{1.4}

	\maketitle

	\abstract{We describe a family of indefinite theta functions of signature $(1,1)$ that can be expressed in terms of trace functions of vertex algebras built from cones in lattices. The family of indefinite theta functions considered has interesting connections with mock theta functions and Appell-Lerch sums. We use these relations to write the McKay-Thompson series of umbral moonshine at lambency $\ell=8,12,16$ in terms of trace functions of vertex algebras modules, and thereby provide the modules for these instances of umbral moonshine. 
	}

	\newpage 
	\tableofcontents

\section{Introduction}
	\label{intro}
Mock theta functions \cite{lostnote} were introduced by Ramanujan in 1920 in his deathbed letter to Hardy \cite{hist}, in which he constructed 17 examples and provided a series of identities satisfied by them. The mysterious nature of these functions, together with the lack of explanation on how he discovered those examples (see \cite{andrews2020ramanujan} for an interesting account), turned the subject into one that has fascinated mathematicians for over a century. %A modern definition capturing the spirit of Ramanujan's characterization of mock theta functions can be found in \cite{book}: 
%\begin{defn}
%A mock theta function $F$ is a function on the upper-half plane such that:
%\begin{itemize}
%	\item[(i)] There are infinitely many roots of unity $\zeta$ for which $F(\tau)$ grows exponentially
%	as $q=e^{2 \pi i \tau}$ approaches $\zeta$ radially from inside the unit disk. 
%	\item[(ii)] For every root of unity $\zeta$, there exists a (weakly holomorphic) modular
%	form $M_{\zeta}$ and a rational number $\alpha_{\zeta}$ such that $F(\tau)-q^{\alpha_{\zeta}}M_{\zeta}(\tau)$	is bounded as $q \to \zeta$ radially from within the unit disk.
%	\item[(iii)] There does not exist a single (weakly holomorphic) modular form $M$ that
%	satisfies (ii).
%\end{itemize}
%\end{defn} 
The theory underlying mock theta functions remained elusive until it was elucidated by Zwegers in his PhD thesis \cite{zwegersthesis} and they are now understood in the context of harmonic Maass forms \cite{HarmonicMaass}, \cite{book}.  The holomorphic parts of harmonic Maass forms are often referred to as mock modular forms and in this language mock theta functions are essentially mock modular forms with shadows given by theta functions. In recent years, mock theta functions have made their appearance in a variety of fields with numerous applications, see e.g. \cite{book}, \cite{mockthapplications}, \cite{3dmodularity}, \cite{proceeding}. One of the most intriguing appearances of mock theta functions is in the context of Umbral Moonshine \cite{UM}, \cite{UM&Niemeier}, \cite{weight1}. Umbral Moonshine consists of a family of 23 moonshine instances associated to appropriate quotients of the automorphism groups of Niemeier lattices, the 23 even unimodular positive-definite lattices of rank 24 with non-trivial root systems. Given the root system $X$ of a Niemeier lattice $L^{X}$, the umbral group $G^X$ associated to it is given by the quotient of the automorphisms group of $L^X$ by the Weyl group  $W^X$ associated to the root system
\begin{equation}
G^X:=\text{Aut}(L^X)/W^X.
\end{equation}
Following \cite{UM,UM&Niemeier}, we will often refer to  twenty-three instances as the different {\em lambencies} of umbral moonshine. 
To each conjugacy class $[g]$ of $G^X$ is associated a vector-valued mock modular form, the umbral McKay-Thompson series $H^X_g$. 
The umbral moonshine conjecture predicts, for each Neiemeier lattice, the existence of a naturally defined bi-graded $G^X$-module
\begin{equation}
\label{Moonshinemodules}
\check{K}^X:=\bigoplus\limits_{r \in I^X}\bigoplus_{\substack{D \in \mathbb{Z}, D\leq 0 \\ D= r^2 \mod 4m}}\check{K}^X_{r, -D/4m}
\end{equation}
such that the corresponding McKay-Thompson series $H^{X}_{g}=(H^{X}_{g,r})$ is related to the graded trace of $g$ over $\check{K}$ by
\begin{equation}
H^X_{g,r}(\tau)=-2q^{-\frac{1}{4m}}\delta_{r,1}+\sum_{\substack{D \in \mathbb{Z}, D\leq 0 \\ D= r^2 \mod 4m}}\text{tr}_{\check{K}^X_{r, -D/4m}}(g)q^{-\frac{D}{4m}}
\end{equation}
where $m$ is the Coxeter number of any simple component of the Niemeier root system $X$, and $I^X \subset \mathbb{Z}/2m \mathbb{Z}$ is specified by 
\begin{equation}
I^X:=
\begin{cases}
\{1,2,3,\ldots, m-1\} & \text{if $X$ has an A-type component,} \\
\{1,3,5, \ldots ,\frac{m}{2}\} & \text{if $X$ has a D-type component and no A-type components,}\\
\{1,4,5\} & \text{if $X=E^4_6$,}\\
\{1,7\} & \text{if $X=E^3_8$.}
\end{cases}
\end{equation}
The existence of the modules \eqref{Moonshinemodules} has been proven in \cite{Mathieuproof} for the case of Mathieu Moonshine, and then in \cite{UMproof} for the remaining cases. These proofs, however, do not prescribe how such modules can be built nor offer much insight on possible further algebraic structure. \\
Since Monstrous Moonshine \cite{MonstrousMoonshine}, the very first example of a moonshine phenomenon, vertex operator algebras have proven to be an invaluable tool to understand the underlying structure behind the moonshine properties \cite{frenkel1}, \cite{frenkel2}, \cite{frenkel3}, \cite{borcherdsproof}. It is thus natural to ask if a similar approach could provide interesting insights in the case of umbral moonshine. For some instances of umbral moonshine it has already been shown that suitable (super) vertex operator algebras can be used to explicitly construct the modules $\check{K}^X$  \cite{moduleE8}, \cite{K3&UMmodule} or to solve the so called ``meromorphic module problem", i.e. building modules such that specific trace functions give the meromorphic Jacobi forms associated to the $H_g$ of Umbral Moonshine \cite{supervertexmeromorphic}, \cite{modulesDtype}, \cite{modulesubgroupMathieu}. In particular, in \cite{moduleE8} the authors built the module $\check{K}^{E^3_8}$ through the means of particular vertex operator algebras obtained from lattice cones. Their construction makes use of the relations between the umbral McKay-Thompson series for $E^3_8$, the fifth order Ramanujan's mock theta functions $\phi_0$, $\phi_1$ and their expressions in terms of indefinite theta functions.
\begin{comment}
 This kind of structure is not unique to the $E^3_8$ case, and in fact mock theta functions play a fundamental role in the specification of the umbral McKay-Thompson series in \cite{UM&Niemeier}.  Actually, using a series of multiplicative relations, a large number of umbral McKay-Thompson series can be specified in terms of such mock theta functions. As all the relevant mock theta functions can be expressed in terms of indefinite theta functions \cite{book},
\end{comment} 
It is natural to ask if the techinques of \cite{moduleE8} can be extended to build modules for other instances of umbral moonshine. 
In this work we employ a particular class of  cone vertex algebras and construct modules for instances of umbral moonshine corresponding to root systems $A^2_7D^2_5$, $A_{11}D_7E_6$, $A_{15}D_9$. In order to achieve this, we will establish intermediate results relating cone vertex algebras to indefinite theta functions that are mock theta functions. In particular, we first describe a specific family of indefinite theta functions can be expressed in terms of trace functions of cone vertex algebras. Then, expressing the umbral McKay-Thompson series $H^X_{g}$ in terms of indefinite theta functions, we relate $H^X_{g}$ to suitable linear combinations of  the traces of cone vertex algebra and other known (super) vertex operator algebras.
In the cases considered, we find that the respective umbral groups act trivially on the underlying cone vertex algebra modules. Thus the modules realizing the McKay-Thompson series appearing in these examples have the structure of a tensor product $R\otimes M$ of a finite group representation $R$ and a (super) vertex algebra module $M$.  In particular, the umbral finite group $G$ acts on $R\otimes M$ as $G\otimes {\bf 1}_{\cal V}$, while the vertex algebra ${\cal V}$ acts as ${\bf 1}_G \otimes {\cal V}$.  
This makes the analysis particularly simple as the representation of the umbral group can be determined independently from the relevant cone vertex algebra structure. \\ As an intermediate result, we also show that the following Appell-Lerch sums
\begin{align}
    \label{AL1}
    &\mu(z_1, z_2; \tau):=\frac{y_1^{\frac{1}{2}}}{\theta(z_2; \tau)}\sum\limits_{n \in \mathbb{Z}}\frac{(-1)^n y^n_2 q^{\frac{n(n+1)}{2}}}{1-y_1q^n},\\
    \label{AL2}
    &\mu_{m,0}(z,\tau):=\sum\limits_{k\in \mathbb{Z}}y^{2km}q^{mk^2}\frac{yq^k+1}{yq^k-1},
\end{align}
admit an expression in terms of indefinite theta functions and cone vertex algebra characters. These are distinguished examples connecting cone vertex algebras to mock theta functions and umbral moonshine. In fact, all Ramanujan mock thetas can be expressed in terms of \eqref{AL1} \cite{zwegersthesis}, \cite{gordon2012survey}, \cite{hickerson2014hecke}, \cite{book}, while \eqref{AL2} appears in the construction of the optimal meromorphic Jacobi forms associated to the umbral McKay-Thompson series \cite{UM&Niemeier}, \cite{UMproof}. The latter fact allows us to draw a connection between the construction of modules for the McKay-Thompson series (as considered in this paper) and the meromorphic module problem considered in \cite{supervertexmeromorphic}, \cite{modulesubgroupMathieu}. Furthermore, the specialized Appell-Lerch sum \eqref{AL2} is also interesting because it captures the non-modular part of the elliptic genus of non-compact supersymmetric coset models, as featured in \cite{Eguchi:2010cb},  \cite{noncompactEG}, \cite{UM&K3}. The techniques used in this paper can be easily used to build an alternative module for the elliptic genus of such theories in terms of cone vertex algebras trace functions.\\ 
The paper is organized as follows: in section 2 we recall basic notions and notations of cone vertex algebras and indefinite theta functions that will be used in the rest of the paper; in section 3 we present a core result of the paper (Theorem \ref{ITtrace}) relating trace functions of cone vertex algebras to indefinite theta functions; in section 4 we give expression specifying the umbral McKay-Thompson series for lambencies $\ell=8,12,16$ in terms of indefinite theta functions and modular forms; finally, in section 5 we specify the umbral McKay-Thompson series considered in section 4 in terms of trace functions of vertex algebra modules (Theorems \ref{l8th}, \ref{l12th}, \ref{l16th}). 
\section{Background}
\subsection{Lattice Vertex Algebras}
\label{sec:coneVOA}
In this subsection we will briefly summarize the construction of vertex algebras associated to lattices,   closely following the exposition in \cite{moduleE8}. The main goal is to introduce the notation and conventions that will be used in the rest of the paper. More details can be found in, e.g, \cite{frenkel3}, \cite{VA&Kac-Moody&Monster},  and \cite{introVOA}.\\
Consider a lattice $L$. 
Let's define $\frak{h}:=L\otimes_{\mathbb{Z}} \mathbb{C}$ with the symmetric $\mathbb{C}$-bilinear form $\braket{\cdot, \cdot}$ naturally inherited from the bilinear form on $L$. Given a formal variable $t$, define $\hat{\frak{h}}:=\frak{h}[t, t^{-1}]\oplus \mathbb{C}{\bf c}$ with the Lie algebra structure given by $[u \otimes t^{m}, v\otimes t^{n}] =m \braket{u,v}\delta_{m+n,0}{\bf c}$ with ${\bf c}$ a central element. The algebra $\hat{\frak{h}}$ has a natural decomposition as $\hat{\frak{h}}=\hat{\frak{h}}^{-}\oplus \hat{\frak{h}}^{0}\oplus \hat{\frak{h}}^{+}$ with $\hat{\frak{h}}^{\pm}:=\frak{h}[t^{\pm 1}]t^{\pm}$ and $\frak{h}^{0}:=\frak{h}\oplus \mathbb{C}{\bf c}$. Given an ordered integral basis $\{\epsilon_j\}$ for the lattice $L$, define 
\begin{equation}
b(\epsilon_i, \epsilon_j)=
\begin{cases}
0 & \text{ if } i \leq j\\
1 & \text{ if } i>j 
\end{cases}
\end{equation}
extended linearly over $L$, and set $\beta(\lambda, \mu):=(-1)^{b(\lambda, \mu)}$. We then consider the ring $\mathbb{C}_{\beta}[L]$ generated by ${\bf v}_{\lambda}$, $\lambda \in L$, satisfying  ${\bf v}_{\lambda}{\bf v}_{\mu}=\beta(\lambda, \mu){\bf v}_{\lambda+\mu}$. Give $\mathbb{C}_{\beta}[L]$ a $\hat{\frak{h}}^0\oplus \hat{\frak{h}}^{+}$-module structure by setting, for $ h\in \frak{h}$ and $\lambda\in L$, $\bold{c}\bold{v}_{\lambda}=\mathbf{v}_{\lambda}$ and $u(m)\bold{v}_{\lambda}=\delta_{m,0}\braket{u, \lambda}\bold{v}_{\lambda}$, where we have used the standard notation $u(m)=u\otimes t^{m}$. Finally, we consider the module
\begin{equation}
V_L:=U(\hat{\frak{h}})\otimes_{U(\hat{\frak{h}}^0\otimes{\hat{\frak{h}}^{+})}}\mathbb{C}_{\beta}[L]. 
\end{equation}
We can equip this module with a (unique) vertex algebra structure with vacuum vector $1\otimes \bold{v}_{0}$, vertex operator map $Y:V_L\to ({\rm End }V_L)[[z,z^{-1}]]$ given by, for $u \in \frak{h}$ and $\lambda \in L$,
\begin{equation}
\begin{split}
&Y(u(-1)\otimes \bold{v}_0, z)=\sum\limits_{n \in \mathbb{Z}}u(n)z^{-n-1}\\
&Y(\bold{1}\otimes \bold{v}_{\lambda},z)=\text{exp}\left(-\sum\limits_{n<0}\frac{\lambda(n)}{n}z^{-n}\right)\text{exp}\left(-\sum\limits_{n>0}\frac{\lambda(n)}{n}z^{-n}\right)\bold{v}_{\lambda}z^{\lambda(0)}
\end{split}
\end{equation}
where $\bold{v}_{\lambda}$ in the right hand side denotes the operator $p \otimes \mathbf{v}_{\mu}\to \beta(\lambda, \mu)p\otimes \mathbf{v}_{\lambda+\mu}$, and $z^{\lambda(0)}(p\otimes \mathbf{v}_{\mu}):=(p\otimes \mathbf{v}_{\mu})z^{\braket{\lambda, \mu}}$.  Furthermore, given the basis $\{\epsilon_j\}$ for $L$ and the dual basis $\{\epsilon'_j\}$, $\epsilon'_j\in L \otimes_{\mathbb{Z}}\mathbb{Q}$ satisfying $\braket{\epsilon'_{i}, \epsilon_j}=\delta_{ij}$, we can define the conformal element 
\begin{equation}
\omega:=\frac{1}{2}\sum\limits_{i}\epsilon'_i(-1)\epsilon_i(-1)\otimes \bold{v}_0. 
\end{equation}
Writing $Y(\omega, z)=\sum \limits_{n \in \mathbb{Z}}L(n)z^{-n-2}$, we have $[L(0), v(n)]=-nv(n)$ and $L(0) 1\otimes\mathbf{v}_{\lambda}= \frac{\braket{\lambda, \lambda}}{2} 1\otimes \mathbf{v}_{\lambda}$. 
In particular, when the bilinear form on $L$ is positive definite, this give $V_L$ the structure of a vertex operator algebra. In the more general case, vector of zero length give infinite dimensional eigenspaces for $L(0)$. We can define a finite order automorphism of $V_L$ by choosing $h\in L\otimes_\ZZ \QQ$ acting as  $h(0)p\otimes\bold{v}_\lambda=\braket{h, \lambda}p\otimes \bold{v}_{\lambda}$ with $p\in S(\hat{\frak{h}}^{-})$ and defining
\begin{equation}
g_h:=e^{2 \pi i h(0)}. 
\end{equation}
In order to build twisted modules for the lattice vertex algebra , let's consider  $\mathbb{C}_{\beta}[L+h]$  generated by $\bold{v}_{\mu+h}$, with $\mu \in L$ and $h \in  L\otimes_{\mathbb{Z}} \mathbb{Q}$, equipped with the $\mathbb{C}_{\beta}[L]$-module structure given by $\bold{v}_{\lambda}\bold{v}_{\mu+h}=\beta(\lambda, \mu)\bold{v}_{\lambda+\mu+h}$ and the $U(\hat{\frak{h}}^0\otimes{\hat{\frak{h}}^{+})}$-module structure $\bold{c}\bold{v}_{\mu+h}=\bold{v}_{\mu+h}$, $u(m)\bold{v}_{\mu+h}=\delta_{m,0}\braket{u, \mu+h}\bold{v}_{\mu+h}$ for $u\in \frak{h}$, $\mu,\lambda \in L$. We can then define $g_h$-twisted modules for the lattice vertex algebra $V_L$ by setting $V_{L+h}:=U(\hat{\frak{h}})\otimes_{U(\hat{\frak{h}}^0\otimes{\hat{\frak{h}}^{+})}}\mathbb{C}_{\beta}[L+h]$ and defining $Y_{h}:=V_L\to(\text{End}V_{L+h})[[z, z^{-1}]]$ similarly as before but with $\bold{v}_{\lambda}$ acting as $\bold{v}_{\lambda}(p\otimes \mathbf{v}_{\mu + h})=\beta(\lambda, \mu)p\otimes \bold{v}_{\lambda+\mu+h}$. When $h$ belongs to the dual lattice $L^{*}=\{\lambda \in L\otimes_{\mathbb{Z}}{Q}|\braket{\lambda, L}\in \mathbb{Z}\}$, the modules are untwisted. Furthermore, all the $g_h$-twisted modules of $V_L$ are given by $V_{L+h'}$ for some $h'\in L\otimes_{\mathbb{Z}} \mathbb{Q}$ congruent to $h$ modulo $L^*$. The action of $L\otimes_{Z}\mathbb{Q}$ on $V_L$ specified by $g_h$ extends to $g_{h^{'}}$-twisted modules through
\begin{equation}
\label{auth}
g_h(p\otimes \bold{v}_{\lambda+h^{'}}):=e^{2 \pi i \braket{h, \lambda}}p\otimes \bold{v}_{\lambda+h^{'}}.
\end{equation} \\
In order to include vertex algebras associated to cones, as opposed to the full lattice, we will describe a family of sub-vertex algebras of $V_L$.
For a $K\subset L$ that is closed under addition that contains 0, the submodule $V_K$ of $V_L$ generated by $\bold{v}_{\lambda}$ for $\lambda \in K$ has the structure of a sub-vertex algebra of $V_L$ with the same conformal element. Furthermore, given $\gamma \in L\otimes_\ZZ \QQ$,  for any $K'\subset L+\gamma$ such that $K+K'\subset K'$, the corresponding $V_{K^{'}}$ with the restriction of the vertex operators $a\otimes b \mapsto Y(a,z)b$ to $V_K \otimes V_{K^{'}}$ has the structure of a twisted module over $V_K$.
\subsection{Indefinite Theta Functions} 
\label{ITbackground}
In this subsection we will review some basic properties of indefinite theta functions of lattices with signature $(r-1, 1)$. We will mostly follow the exposition in chapter 8 of \cite{book}.\\ We will start by introducing some notation. For the rest of the paper we will set $q:=e^{2 \pi i \tau}$ and $y:=e^{2\pi i z}$. Given a symmetric matrix $A$ with integer coefficients, we define the bilinear form $B(\bold{v}, \bold{w}):=\bold{v}^{T}A\bold{w}$ and, correspondingly, the quadratic form $Q(\mathbf{v}):=\frac{1}{2}B(\bold{v}, \bold{v})$. It is a well known fact that, when $Q$ is positive definite, given $\bold{x}_0\in \mathbb{Z}^r$ we have that 
\begin{equation}
\label{Jacobitheta}
\Theta_{Q, \bold{x}_0}(\tau;z):=\sum\limits_{\bold{n} \in \mathbb{Z}^{r}}q^{Q(\bold{n})}y^{B(\bold{n}, \bold{x}_0)}
\end{equation}
is a  Jacobi form of weight $\frac{r}{2}$ and index $Q(\bold{x}_0)$ \cite{EZ}. This result does not hold when $Q$ is not positive definite. When the quadratic form is of signature $(r-1, 1)$, i.e. when the largest linear subspace on which $Q$ is negative definite has dimension $1$, generalisations of \eqref{Jacobitheta} were studied by Zwegers \cite{zwegersthesis}. For such quadratic forms, the set $\{\bold{c}\in \mathbb{R}^r: Q(\bold{c})<0\}$ splits into two connected components, we fix one of these and denote it $C_Q$. Explicitly we choose a $\bold{c}_0$ such that $Q(\bold{c}_0)<0$ and define
\begin{equation}
C_{Q}:=\{\bold{c}\in \mathbb{R}^r: Q(\bold{c})<0, B(\bold{c}, \bold{c}_0)<0\}.
\end{equation} 
We also define 
\begin{equation}
S_{Q}:=\{\bold{c}=(c_1,\ldots, c_r)\in \mathbb{Z}^r: \text{gcd}(c_1, \ldots, c_r)=1, Q(\bold{c})=0, B(\bold{c}, \bold{c}_0)<0\}
\end{equation}
and consider the compactification of $C_Q$, $\overline{C}_Q:=C_Q\cup S_Q$. We furthermore define, $\forall \bold{c}\in \overline{C}_Q$,
\begin{equation}
R(\bold{c}):=
\begin{cases}
\mathbb{R}^{r} & \text{if } \bold{c}\in C_Q,\\
\{\bold{a}\in \mathbb{R}^r: B(\bold{a}, \bold{c})\not \in \mathbb{Z}\} & \text{if } \bold{c}\in S_Q.
\end{cases}
\end{equation}
With the above notation, given a symmetric matrix $A$,  $\bold{c}_1,\bold{c}_2 \in \overline{C}_{Q}$,  $\bold{a}\in R(\bold{c}_1)\cap R(\bold{c}_2)$, and $\bold{b}\in \mathbb{R}^{r}$, we can define the indefinite theta functions
\begin{equation}
\label{it}
\Theta^{(\rho)}_{\bold{a}, \bold{b}}(\tau):=\sum\limits_{\bold{n}\in \bold{a}+ \mathbb{Z}^r}\rho(\bold{n}; \tau)e^{2\pi i B(\bold{b}, \bold{n})}q^{Q(\bold{n})}
\end{equation}
where we have written 
\begin{equation}
\rho(\bold{n}; \tau)=\rho^{\bold{c}_1}(\bold{n}; \tau)-\rho^{\bold{c}_2}(\bold{n}; \tau)
\end{equation}
with 
\begin{equation}
\rho^{\bold{c}}(\bold{n}; \tau):=
\begin{cases}
E\left(\frac{B(\bold{c}, \bold{n})v^{\frac{1}{2}}}{\sqrt{-Q(\bold{c})}}\right) & \text{if } \bold{c}\in C_Q\\
\text{sgn}(B(\bold{c}, \bold{n})) & \text{if } \bold{c} \in S_Q
\end{cases}
\end{equation}
in which $v=\text{Im}(\tau)$ and $E$ is the error function
\begin{equation}
E(z):=2\int_{0}^{z}e^{-\pi t^2}dt.
\end{equation}
It has been shown \cite{zwegersthesis} that, with the assumptions above, the series in \eqref{it} converges absolutely. Furthermore, it has been shown  \cite{zwegersthesis}, \cite{book} that for $\bold{c}_1, \bold{c}_2 \in \mathbb{Z}^r\cap \overline{C}_Q$ with relatively prime coordinates,  $\bold{a}, \bold{b} \in R(\bold{c}_1)\cap R( \bold{c}_2)$,   the indefinite theta in \eqref{it} is a component of a vector-valued mixed harmonic Maass form of weight $\frac{r}{2}$ for {SL$_2(\ZZ)$}, with holomorphic part given by
\begin{equation}
\label{holIT}
\Theta^{+}_{\bold{a}, \bold{b}}(\tau)=\sum\limits_{\bold{n} \in \mathbb{Z}^r +\bold{a}}\left[\text{sgn}(B(\bold{c}_1, \bold{n}))-\text{sgn}(B(\bold{c}_2, \bold{n}))\right]e^{2\pi i B(\bold{b}, \bold{n})}q^{Q(\bold{n})}.
\end{equation}
It is also  shown \cite{zwegersthesis}, \cite{gordon2012survey}, \cite{hickerson2014hecke}, \cite{book} that all Ramanujan's mock theta functions (and a further number of mock theta functions discovered later) can be written in terms of a linear combination of modular forms and indefinite theta functions \eqref{holIT} with $r=2$. This can be viewed as a generalization of the following relation between the indefinite theta functions and the Appell-Lerch sum
\begin{equation}
\label{ALsum}
\mu(z_1, z_2; \tau):=\frac{y_1^{\frac{1}{2}}}{\theta(z_2; \tau)}\sum\limits_{n \in \mathbb{Z}}\frac{(-1)^n y^n_2 q^{\frac{n(n+1)}{2}}}{1-y_1q^n},
\end{equation}
where $z_1,z_2 \in \mathbb{C}/(\mathbb{Z}+\tau\mathbb{Z})$, $y_j=e^{2 \pi i z_j}$ for $j=1,2$, and $\theta(z;\tau)$ is the Jacobi theta function
\begin{equation}
\theta(z; \tau):=\sum \limits_{n \in \mathbb{Z}+\frac{1}{2}}e^{\pi i n^2\tau+2\pi i n \left(z+\frac{1}{2}\right)}.
\end{equation}
Namely, defining 
\begin{equation}
\label{ztheta}
\Theta^{+}_{A, \bold{c}_1, \bold{c}_2}(\bold{z};\tau):=\sum\limits_{n\in \mathbb{Z}^r}\left[\text{sgn}\left(B\left(\bold{c}_1, \bold{n}+\frac{\text{Im}(\bold{z})}{\text{Im}(\tau)}\right)\right)-\text{sgn}\left(B\left(\bold{c}_2, \bold{n}+\frac{\text{Im}(\bold{z})}{\text{Im}(\tau)}\right)\right)\right]e^{2\pi i B(\bold{z}, \bold{n})}q^{Q(\bold{n})}
\end{equation}
it can be shown \cite{book} that, for $0<\frac{\text{Im}(z_1)}{\text{Im}(\tau)}, \frac{\text{Im}(z_2)-\text{Im}(z_1)}{\text{Im}(\tau)}+\frac{1}{2}<1$, we have the following relation
\begin{equation}
\label{mutheta}
\mu(z_1, z_2; \tau)=\frac{y_1^{\frac{1}{2}}}{2\theta(z_2;\tau)}\Theta^{+}_{A,\bold{c}_1, \bold{c}_2}\left(z_1, z_2-z_1+\frac{\tau+1}{2}; \tau\right)
\end{equation}
for $A=\left(\begin{smallmatrix}1 & 1 \\ 1 & 0\end{smallmatrix}\right)$, $\bold{c}_1=(0,1)$, $\bold{c}_2=(-1,1)$.\\
Notice that, writing $\bold{z}=\bold{a}\tau+\bold{b}$ with $\bold{a}\in  R(\bold{c}_1)\cap R(\bold{c}_2), \bold{b} \in \mathbb{R}^r$, equation \eqref{ztheta} can be related to expression \eqref{holIT} through
\begin{equation}
\label{zabrel}
\Theta^+_{A,{\bf c}_1,{\bf c}_2}({\bf a}\tau+{\bf b},\tau)=e^{-2\pi i B({\bf a},{\bf b})}q^{-Q({\bf a})}\Theta^{+}_{{\bf a},{\bf b}}(\tau).
\end{equation}
The relation to Ramanujan's mock theta functions follows via their expression in terms of the universal mock theta functions $g_2$, $g_3$  \cite{universalmock} which in turn can be related to the Appell-Lerch sum \eqref{ALsum} \cite{kang}.

\section{Indefinite Theta Functions and Cone Vertex Algebras}
\label{sec:IT&coneVOA}
In this section we will describe a family of trace functions of vertex algebras modules that can be expressed in terms of indefinite theta functions. \\
Consider a symmetric $2\times2$ matrix $A$ with integer coefficients, the associated bilinear and quadratic forms $B$ and $Q$ as in section \ref{ITbackground},  and the vectors $\bold{c}_1,\bold{c}_2\in\bar{C}_{Q}$ satisfying 
\begin{equation}
\label{ccondition}
\begin{array}{ll}
\bold{c}_1^{\rm T}A=k(1,0),&
\bold{c}_2^{\rm T}A=k'(0,-1),
\end{array}
\end{equation}
with $k, k'\in \mathbb{R}^{*}$ and $\text{sgn}(k)=\text{sgn}(k')$.
With the above constraints, we will consider the family of rank 2 indefinite theta functions \footnote{From now on we will omit the $+$ apex from the symbol $\Theta^{+}_{\mathbf{a}, \mathbf{b}}$}
\begin{equation}
\label{indtheta}
\Theta_{{\bf a},{\bf b}}(N\tau)=\sum\limits_{{\bf n} \in \mathbb{Z}^2+\bold{a}} \left[\text{sgn}(B(\bold{c}_1, \bold{n}))-\text{sgn}(B(\bold{c}_2, \bold{n}))\right]e^{2\pi i B({\bf n},{\bf b})}q^{NQ({\bf n})}
\end{equation}
with $N$ a positive integer, $\bold{b}\in \mathbb{R}^2$, and $\bold{a}=(a_1, a_2) \in \mathbb{Q}^2\cap R(\bold{c}_1)\cap R(\bold{c_2})$. As we will show in later sections, in the cases considered the components of umbral McKay-Thompson series can be rewritten in terms of indefinite theta functions with such quadratic form $A$ and vectors $\bold{c}_1, \bold{c}_2$. \\ Let's now define the relevant cone vertex algebras. Following the construction in section \ref{sec:coneVOA}, we start by defining the underlying lattice. We consider the rank 2 lattice $L^{(N)}$ generated by $\epsilon_1,\epsilon_2$ with the bilinear form $\braket{\cdot, \cdot}$ specified by the matrix $A$ and a positive integer $N$ as
\begin{equation}
\braket{\epsilon_i, \epsilon_j}=NA_{ij}. 
\end{equation}
%for a fixed positive integer $N$.
\begin{comment}
 such that the lattice is even \GS{this requirement might be not necessary}, i.e. $NA_{ii}\in 2\mathbb{Z}$ for $i=1,2$. 
 \end{comment}
Consider the sublattice of $L^{(N)}$ given by the cone 
\begin{equation}
\label{cone}
	P^{(N)}=\{\sum\limits_{i=1}^2\alpha_i\epsilon_i \in L^{(N)} \otimes \mathbb{Q}: \alpha_i\geq0, \forall i=1,2\}
\end{equation}
and its shifted version $P^{(N)}+\gamma:=\{\mu+\gamma|\mu\in P^{(N)}\}$. As described in section \ref{sec:coneVOA}, $V_{P^{(N)}}$, generated by $\bold{v}_{\lambda}$ for $\lambda \in P^{(N)}$,  is a sub-vertex operator algebra of $V_{L^{(N)}}$. For $\mathbf{a}:=(a_1,a_2) \in \mathbb{Q}^2$ let's define $\rho^{+}_{\mathbf{a}}:=a_1\epsilon_1+a_2\epsilon_2$ and $\rho^{-}_{\mathbf{a}}:=(1-a_1)\epsilon_1+(1-a_2)\epsilon_2$.  
To any lattice $L^{(N)}$ we thus associate a module $V_{\mathbf{a}}^{(N)}$ given by the following direct sum
\begin{equation}
	V_{\mathbf{a}}^{(N)}:=V_{P^{(N)}+\rho^{+}_{\mathbf{a}}}\oplus V_{P^{(N)}+\rho^{-}_{\mathbf{a}}}
\end{equation}
where $V_{P^{(N)}+\rho^{+}_{\mathbf{a}}}$ and $V_{P^{(N)}+\rho^{-}_{\mathbf{a}}}$ are the modules of the  vertex algebra $V_{P^{(N)}}$  built from $P^{(N)}+\rho^{+}_{\mathbf{a}}$ and $P^{(N)}+\rho^{-}_{\mathbf{a}}$ respectively. Notice that, when $\rho^{\pm}_{\bold{a}}\in {L^{(N)}}^*$, as is the case when  $a_1, a_2\in \left\{0, \frac{1}{N}, \frac{2}{N}, \ldots, \frac{N-1}{N}\right\}$, the modules $V_{P^{(N)}+\rho^{\pm}_{\mathbf{a}}}$ are untwisted.  For $\lambda=n_1\epsilon_1+n_2\epsilon_2 + \rho^{\pm}_{\mathbf{a}}\in P^{(N)}+\rho^{\pm}_{\mathbf{a}}$, $\mathbf{b}=(b_1,b_2) \in \mathbb{Q}^2$ and write $\mathbf{n}=(n_1,n_2)$, let's define the operator $g_{\mathbf{b}}: V_{\mathbf{a}}^{(N)} \to V_{\mathbf{a}}^{(N)}$ acting as 
\begin{equation}
\label{twisting}
g_{\bf{b}}(p\otimes \mathbf{v}_{\lambda}):=\begin{cases}
	e^{2 \pi i B\left({\bf n}, {\bf b}\right)}p\otimes \mathbf{v}_{\lambda} & \text{if } \mathbf{v}_\lambda \in V_{P^{(N)}+\rho^{+}_{\mathbf{a}}} \\
	-e^{-2 \pi i B\left({\bf n}+\mathbf{1}, {\bf b}\right)}p\otimes \mathbf{v}_{\lambda} & \text{if } \mathbf{v}_{\lambda} \in V_{P^{(N)}+\rho^{-}_{\mathbf{a}}}
\end{cases}
\end{equation}
The main object we will be interested in is the trace function
\begin{align}
	\label{gradtrace}
	T^{(N)}_{\bold{a},{\bf b}}(\tau)&:=\text{Tr}_{V_{\mathbf{a}}^{(N)}}\left(g_{\bf{b}} q^{L(0)-c/24}\right).
\end{align}
Notice that, in general, when there are non-trivial vectors with non-positive norm in $P+\rho^{\pm}_{\bold{a}}$ the trace will not converge. We will thus restrict to matrices $A$ that are positive definite on $P+\rho^{\pm}_{\bold{a}}$, i.e. $v^TAv>0$ $\forall v \in P+\rho^{\pm}_{\bold{a}}$.  Under such assumptions, we will now show how the trace functions $T^{(N)}_{{\bf a},{\bf b}}(\tau)$ are related to $\Theta_{{\bf a},{\bf b}}(N\tau)$.

\begin{lem}
	Let $\bold{a}=(a_1,a_2)\in \mathbb{Q}^2$ with $0<a_1,a_2<1$, $A$ a symmetric matrix positive definite on $P^{(N)}+\rho^{\pm}_{\bold{a}}$, and $\mathbf{c}_1, \mathbf{c}_2 \in \overline{C}_Q$ such that 
	\begin{equation}
		\begin{array}{ll}
			\bold{c}_1^{\rm T}A=k(1,0),&
			\bold{c}_2^{\rm T}A=k'(0,-1),
		\end{array}
	\end{equation}
	for some $k, k'\in \mathbb{R}^*$ with ${\rm sgn}(k)={\rm sgn}(k')$. We have 
	\begin{equation}
		\label{T}
		T^{(N)}_{\bold{a},{\bf b}}(\tau)={\rm sgn}(k)\frac{e^{-2 \pi i B(\bold{a}, {\bf b})}}{2\eta(\tau)^2}\Theta_{\bold{a}, {\bf b}}(N\tau),
	\end{equation}

\end{lem}
\begin{proof}
	Explicitly, \eqref{gradtrace}  equals 
	\begin{equation}
		\label{eq:explicitchar}
		\begin{split}
			T^{(N)}_{\bold{a},{\bf b}}(\tau)&=\frac{1}{\eta(\tau)^2}\left(\sum\limits_{\mu \in P^{(N)}+\rho^{+}_{{\bf a}}}e^{2 \pi i B({\bf n}, {\bf b})}q^{\frac{\braket{\mu,\mu}}{2}}-\sum\limits_{\mu \in P^{(N)}+\rho^{-}_{{\bf a}}}e^{-2 \pi i B({\bf n}+\mathbf{1}, {\bf b})}q^{\frac{\braket{\mu,\mu}}{2}}\right)\\
			&=\frac{1}{\eta(\tau)^2}\left(\sum_{\substack{\mathbf{n}\in \mathbb{Z}^2 \\ n_1,n_2\geq 0}}e^{2 \pi i B({\bf n}, {\bf b})}q^{Q(\mathbf{n}+\mathbf{a})}-\sum_{\substack{\mathbf{n}\in \mathbb{Z}^2 \\ n_1,n_2\geq 0}}e^{-2 \pi i B({\bf n}+\mathbf{1}, {\bf b})}q^{Q(\mathbf{n}+\mathbf{1}-\mathbf{a})}\right)\\
			&=\frac{1}{\eta(\tau)^2}\left(\sum_{\substack{\mathbf{n}\in \mathbb{Z}^2 \\ n_1,n_2\geq 0}}e^{2 \pi i B({\bf n}, {\bf b})}q^{Q(\mathbf{n}+\mathbf{a})}-\sum_{\substack{\mathbf{n}\in \mathbb{Z}^2 \\ n_1,n_2\leq 0}}e^{2 \pi i B({\bf n}-\mathbf{1}, {\bf b})}q^{Q(\mathbf{n}-\mathbf{1}+\mathbf{a})}\right) \\
			&=\frac{1}{\eta(\tau)^2}\left(\sum_{\substack{\mathbf{n}\in \mathbb{Z}^2 \\ n_1,n_2\geq 0}}-\sum_{\substack{\mathbf{n}\in \mathbb{Z}^2 \\ n_1,n_2< 0}}\right) e^{2 \pi i B({\bf n}, {\bf b})}q^{Q(\mathbf{n}+\mathbf{a})}
		\end{split}
	\end{equation}
	where we have written $\mu=(n_1+a_1)\epsilon_1+(n_2+a_2)\epsilon_2$ and $\bold{n}=(n_1,n_2)$.\\
	On the other hand, since $0<a_1,a_2<1$,  the factor $\rho^{{\bf c},{\bf c'}}({\bf n})$ in \eqref{indtheta} equals, using \eqref{ccondition},
	\begin{equation}
		\begin{split}
			\text{sgn}(B({\bf c},{\bf n}))-\text{sgn}(B({\bf c'},{\bf n}))&=\text{sgn}(k)\text{sgn}(n_1+a_1)+\text{sgn}(k')\text{sgn}(n_2+a_2)\\
			&=
			\begin{cases}
				2 \text{sgn}(k) & \text{if } n_1,n_2\geq0 \\
				-2 \text{sgn}(k) & \text{if } n_1,n_2 <0 \\
				0 & \text{otherwise}
			\end{cases}
		\end{split}
	\end{equation}
	where we have also used $\text{sgn}(k)=\text{sgn}(k')$.
	By comparison it is immediate to see that the difference between the two sums in \eqref{eq:explicitchar} equals the indefinite theta function defined in \eqref{indtheta} up to the overall $\frac{e^{-2 \pi i B(\bf{a}, {\bf b})}}{2}\text{sgn}(k)$ factor. 
\end{proof} 
The result can be easily generalized to the cases where  $a_1$ or $a_2$ is equal to 0. In that case we will have an extra one-dimensional theta series appearing in the right hand side of \eqref{T}. In fact, let's consider for example the case $a_1=0, a_2\neq 0$. We have 
\begin{equation}
\begin{split}
\text{sgn}(B({\bf c},{\bf n}))-\text{sgn}(B({\bf c'},{\bf n}))&=\text{sgn}(k)\text{sgn}(n_1)+\text{sgn}(k')\text{sgn}(n_2+a_2)\\
&=
\begin{cases}
2 \text{sgn}(k) & \text{if } n_1\geq 0, n_2>0 \\
-2 \text{sgn}(k) & \text{if } n_1,n_2 <0 \\
\text{sgn}(k) & \text{if } n_1\geq 0, n_2=0\\
-\text{sgn}(k) & \text{if } n_1<0,n_2 =0\\
0 & \text{otherwise}
\end{cases}
\end{split}
\end{equation}
so we have to add some series to account the case $n_2=0$ correctly. An easy calculation for the general case shows that we have the following
\begin{thm}
\label{ITtrace}
Let $\bold{a} = (a_1,a_2)\in \mathbb{Q}^2$ with $0\leq a_1,a_2<1$,  $A$ a symmetric matrix positive definite on $P^{(N)}+\rho^{\pm}_{\mathbf{a}}$, and  $\mathbf{c}_1, \mathbf{c}_2 \in \overline{C}_Q$ such that 
\begin{equation}
\begin{array}{ll}
\bold{c}_1^TA=k(1,0),&
\bold{c}_2^TA=k'(0,-1),
\end{array}
\end{equation}
for some $k, k'\in \mathbb{R}^*$ with $\text{sgn}(k)=\text{sgn}(k')$. We have 
\begin{equation}
\begin{split}
T^{(N)}_{\bold{a},{\bf b}}(\tau)=&\text{sgn}(k)\frac{e^{-2 \pi i B(\bold{a}, {\bf b})}}{2\eta(\tau)^2}\Biggl[\Theta_{\bold{a}, {\bf b}}(N\tau)+\delta_{a_1}\sum_{\substack{n_1=0 \\ n_2\in\mathbb{Z}}}e^{2\pi iB(\bold{n}+\bold{a}, \bold{b})}q^{NQ(\bold{n}+\bold{a})}
\\&+\delta_{a_2}\sum_{\substack{n_1\in\mathbb{Z} \\ n_2=0}}e^{2\pi iB(\bold{n}+\bold{a}, \bold{b})}q^{NQ(\bold{n}+\bold{a})}-\delta_{a_1}\delta_{a_2}e^{2 \pi iB(\bold{a},\bold{b})}q^{Q(\bold{a})}\Biggr],
\end{split}
\end{equation}
where $\delta_{i}$ is the Kronecker delta $\delta_{i,0}$.
\end{thm}

We will now show that the Appell-Lerch sums \eqref{AL1} and \eqref{AL2} can be written in terms of the trace functions \eqref{gradtrace}. These functions will also be important for later sections. Let's consider the Appell-Lerch sum Let's first consider the Appell-Lerch sum \eqref{AL1}. We have the following
\begin{equation}
\mu(z_1, z_2; \tau)=\frac{y_1^{\frac{1}{2}}}{\theta(z_2; \tau)}\sum\limits_{n \in \mathbb{Z}}\frac{(-1)^n y^n_2 q^{\frac{n(n+1)}{2}}}{1-y_1q^n}.
\end{equation}
We have the following
\begin{cor}
Let $\bold{\tilde{a}}=(\tilde{a}_1, \tilde{a}_2)\in \mathbb{Q}^2$ such that $0<\tilde{a}_1<1$, $0\leq \tilde{a}_2-\tilde{a}_1+\frac{1}{2}<1$, $\bold{\tilde{b}}=(\tilde{b}_1, \tilde{b}_2)\in \mathbb{R}^2$, $N\in \mathbb{N}^{*}$ . Let $T^{(N)}_{\bold{a}, \bold{b}}$ be the trace function \eqref{gradtrace} associated to the lattice with quadratic form $N\left(\begin{smallmatrix} 1 & 1\\ 1&0\end{smallmatrix}\right)$. We have 
	 \begin{equation}
	 \mu(\tilde{\bold{a}}N\tau+\tilde{\bold{b}}; N\tau)=\frac{2q^{\frac{N a_1}{2}}\eta(\tau)^2}{\theta((a_2+a_1-1/2)\tau+b_2+b_1-1/2; N\tau)}q^{-NQ(\bold{a})}T^{(N)}_{\bold{a}, \bold{b}}(\tau).
	 \end{equation}
where $\bold{a}:=(a_1,a_2)=\left(\tilde{a}_1, \tilde{a}_2-\tilde{a}_1+\frac{1}{2}\right)$ and $\bold{b}:=(b_1, b_2)=\left(\tilde{b}_1, \tilde{b}_2-\tilde{b}_1+\frac{1}{2}\right)$.
\end{cor}
\begin{proof}
The result follows by the rewriting of $\mu$ in terms of indefinite theta functions. In fact, using equations \eqref{mutheta} and \eqref{zabrel}, we have
\begin{equation}
\mu(\tilde{\bold{a}}\tau+\tilde{\bold{b}}; \tau)=\frac{q^{\frac{a_1}{2}}e^{\pi i b_1}}{\theta((a_2+a_1-1/2)\tau+b_2+b_1-1/2; \tau)}e^{-2\pi i B(\bold{a}, \bold{b})}q^{-Q(\bold{a})}\Theta_{\bold{a}, \bold{b}}(\tau).
\end{equation} 

The choice $A=\left(\begin{smallmatrix}1 & 1 \\ 1 & 0\end{smallmatrix}\right)$, $\bold{c}_1=(0,1)$, $\bold{c}_2=(-1,1)$ satisfies \eqref{ccondition}. Furthermore, while $P^{(N)}$ has infinitely many vectors of the form $n_2\epsilon_2$ $\forall n_2 \in \mathbb{Z}$, that have null norm, the scalar product $\braket{\lambda, \lambda}=N(n_1+a_1)^2+2N(n_1+a_1)(n_2+a_2)$ is strictly positive $\forall \lambda \in P^{(N)}+\rho^{\pm}_{\bold{a}}$ with for $0\leq a_2<1$, $0<a_1<1$. Thus, using Theorem \ref{ITtrace}, the conclusion follows.
\end{proof}
As already mentioned, all Ramanujan's mock theta functions can be written in terms of the Appell-Lerch sum \eqref{ALsum} (up to modular functions) with the choice of $\bold{z}=\bold{\tilde{a}}\tau+\bold{\tilde{b}}$ discussed above \cite{book}, \cite{zwegersthesis}, \cite{gordon2012survey}, \cite{hickerson2014hecke}, thus they can be expressed in terms of cone vertex algebras trace functions using the previous Corollary. 

Let's now consider the specialized Appell-Lerch sum \eqref{AL2}. This function appears in the definition of the meromorphic Jacobi forms associated to the umbral McKay-Thompson series \cite{UM&Niemeier}, \cite{UMproof}, \cite{dabholkar2012quantum}. We will see that it also admits an expression in terms of the trace function \eqref{gradtrace}. Specifically, we have the following
\begin{cor}
\label{specALtrace}
Let $a\in \mathbb{Q}^{*}$, $|a|<1$, $b\in \mathbb{R}$, $N\in \mathbb{N}^{*}$. Consider
the lattice with quadratic form $A=N\left(\begin{smallmatrix} 2m & 1\\ 1&0\end{smallmatrix}\right)$ with $m\in \mathbb{N}$, and   the trace function  \eqref{gradtrace} 
$T^{(N)}_{\bold{a}, \bold{b}}$  associated to it. We have
\begin{equation}
\mu_{m,0}(aN\tau+b,N\tau)=-2f(b)q^{-2mNa^2}\eta(\tau)^2T^{(N)}_{\bold{a}, \bold{b}}(\tau)-\sum\limits_{n\in \mathbb{Z}+a}e^{2\pi i nb}q^{2mNn^2}
\end{equation}
where $\bold{b}=(b,0)$ and $\mathbf{a}=(a,0)$, $f(b)=1$ when $a>0$ while $\mathbf{a}=(1+a,0)$, $f(b)=e^{-4\pi i b}$ when $a<0$.
\end{cor}
\begin{proof}
Let's show first that, for $\left|\frac{\text{Im}(z)}{\text{Im}(\tau)}\right|<1$, $\text{Im}(z)\neq 0$, we can write $\mu_{m,0}(z, \tau)$ in terms of indefinite theta functions satisfying the conditions of Theorem \ref{ITtrace}.
We write $\mu_{m,0}(z, \tau)=f_1(z, \tau)+f_2(z, \tau)$, with 
\begin{equation}
\begin{array}{ll}
f_1(z, \tau):=-\sum\limits_{k\in \mathbb{Z}}\frac{y^{2km}q^{mk^2}}{1-yq^k},&
f_2(z, \tau):=-\sum\limits_{k\in \mathbb{Z}}\frac{y^{2km+1}q^{mk^2+k}}{1-yq^k}
\end{array}
\end{equation}
Let us also set 
\begin{equation}
\begin{array}{lll}
A=\begin{pmatrix}2m & 1 \\ 1 & 0\end{pmatrix}, & 
\bold{c}_1=\begin{pmatrix}
0 \\ 1
\end{pmatrix}, &
 \bold{c}_2=\begin{pmatrix}
-1 \\ 2m
\end{pmatrix}.
	\end{array}
	\end{equation}
	Let us first focus on the domain $0<\frac{{\text{Im}}(z)}{{\text{Im}}(\tau)}<1$. Using the geometric series expansion for the denominator in the range $0<\frac{\text{Im}(z)}{\text{Im}(\tau)}<1$, we can rewrite $f_1$ as 
	\begin{equation}
	\begin{split}
	f_1(z, \tau)&=-\left(\sum\limits_{k,l\geq 0}-\sum\limits_{k,l<0}\right)y^{2km+l}q^{mk^2+kl}\\
	&=-\frac{1}{2}\sum\limits_{(k,l)\in \mathbb{Z}^2}\left[\text{sgn}\left(k+\frac{\text{Im}(z)}{\text{Im}(\tau)}\right)+
	\text{sgn}\left(l\right)\right]e^{2\pi i B[(k,l), (z,0)]}q^{Q((k,l))}-\frac{1}{2}\sum\limits_{k\in \mathbb{Z}}y^{2mk}q^{mk^2}
	\end{split}
	\end{equation}
	where the second sum has to be introduced to fix the contributions for $l=0$.
	It is then immediate to see that we can write
	\begin{equation}
	f_1(z, \tau)=-\frac{1}{2}\Theta_{A, \bold{c}_1, \bold{c}_2}(z, 0;\tau) -\frac{1}{2}\sum\limits_{k\in \mathbb{Z}}y^{2mk}q^{mk^2}.
	\end{equation}
	Analogously
	\begin{equation}
	\begin{split}
	f_2(z, \tau)&=-\left(\sum\limits_{k,l\geq 0}-\sum\limits_{k,l<0}\right)y^{2km+l+1}q^{mk^2+k(l+1)}\\
	&=-\left(\sum\limits_{k\geq 0, l\geq 1}-\sum\limits_{k<0,l<1}\right)y^{2km+l}q^{mk^2+kl}\\
	&=-\frac{1}{2}\sum\limits_{(k,l)\in \mathbb{Z}^2}\left[\text{sgn}\left(k+\frac{\text{Im}(z)}{\text{Im}(\tau)}\right)+
	\text{sgn}\left(l\right)\right]e^{2\pi i B[(k,l), (z,0)]}q^{Q((k,l))}+\frac{1}{2}\sum\limits_{k\in \mathbb{Z}}y^{2mk}q^{mk^2}\\
	&=-\frac{1}{2}\Theta_{A, \bold{c}_1, \bold{c}_2}(z, 0;\tau) +\frac{1}{2}\sum\limits_{k\in \mathbb{Z}}y^{2mk}q^{mk^2}
	\end{split}
	\end{equation}
	where in the second line we have sent $l+1\to l$ and the second sum is again due to the $l=0$ terms. Interestingly, when summing $f_1$ and $f_2$, only the contribution of the indefinite theta survives, and we have 
	\begin{equation}
	\label{specializedALIT}
	\mu_{m,0}(z, \tau)=-\Theta_{A, \bold{c}_1, \bold{c}_2}(z, 0;\tau).
	\end{equation}
\\
In particular, notice that $\bold{c}_1$ and $\bold{c_2}$ satisfy \eqref{ccondition}.   When $z=a\tau+b$ with $a\in \mathbb{Q}$, $0<a<1$, $b\in \mathbb{R}$, for any $N\in \mathbb{N}$, using Theorem \ref{ITtrace}  we have
\begin{equation}
\label{mumcva}
\begin{split}
\mu_{m,0}(aN\tau+b,N\tau)&=-2e^{-4\pi i m ab}q^{-2mNa^2}\Theta_{\mathbf{a},{\bf b}}(N\tau)\\&=-2q^{-2mNa^2}\eta(\tau)^2T^{(N)}_{\bold{a}, \bold{b}}(\tau)-\sum\limits_{n\in \mathbb{Z}+a}e^{2\pi i nb}q^{2mNn^2}
\end{split}
\end{equation}
with $\bold{a}=(a,0)$ and $\bold{b}=(b,0)$. \\
The same result still holds in the domain $0<-\frac{\text{Im}(z)}{\text{Im}(\tau)}<1$.  In this case we have $|yq^k|<1$ for $k>0$ and $|yq^k|>1$ for $k\leq 0$. Thus we get 
\begin{equation}
f_1(z, \tau)=-\left(\sum_{\substack{k>0 \\ l\geq 0}}-\sum_{\substack{k\leq 0 \\ l<0}}\right)y^{2km+l}q^{mk^2+kl}.
\end{equation}
On the other side, in this domain 
\begin{equation}
\text{sign}\left(k+\frac{\text{Im}(z)}{\text{Im}(\tau)}\right)+\text{sign}(l)=
\begin{cases}
2 & \text{if } k>0, l>0, \\
1 & \text{if } k>0, l=0, \\
-1 & \text{if } k\leq 0, l=0,\\
-2 & \text{if } k \leq 0, l<0,\\
0 & \text{otherwise}.
\end{cases}
\end{equation}
So we have again
\begin{equation}
f_1(z, \tau)=-\frac{1}{2}\sum\limits_{(k,l)\in \mathbb{Z}^2}\left[\text{sgn}\left(k+\frac{\text{Im}(z)}{\text{Im}(\tau)}\right)+
\text{sgn}\left(l\right)\right]e^{2\pi i B[(k,l), (z,0)]}q^{Q((k,l))}-\frac{1}{2}\sum\limits_{k\in \mathbb{Z}}y^{2mk}q^{mk^2}.
\end{equation}
Proceeding in the same way for $f_2(z, \tau)$, it is possible to show that equation \eqref{specializedALIT} still holds in the domain $0<-\frac{\text{Im}(z)}{\text{Im}(\tau)}<1$. In particular, we have
\begin{equation}
\mu_{m,0}(a\tau+b,\tau)=-2e^{-4\pi i m ab}q^{-2ma^2}\Theta_{\mathbf{a},{\bf b}}(\tau)
\end{equation}
where $\mathbf{a}=(1+a,0)$ and $\mathbf{b}=(b,0)$ and we have used the property $\Theta_{\mathbf{a}, \mathbf{b}}=\Theta_{\mathbf{a}+\mathbf{s}, \mathbf{b}}$ for all $\mathbf{s}\in \mathbb{Z}^2$. Notice that $1+a>0$ and thus we can use Theorem \ref{ITtrace}.  We then get
\begin{equation}
\mu_{m,0}(aN\tau+b,N\tau)=-2e^{-4\pi i m b}q^{-2mNa^2}\eta(\tau)^2T^{(N)}_{\bold{a}, \bold{b}}(\tau)-\sum\limits_{n\in \mathbb{Z}+a}e^{2\pi i nb}q^{2mNn^2}. 
\end{equation}
\end{proof}

\section{Umbral McKay-Thompson Series, Mock Theta Functions and Indefinite Thetas}
\label{ITspecification}
In this section we will write the umbral McKay-Thompson series appearing for lambency $\ell=8, 12, 16$ in terms of mock theta functions, eta quotients and Jacobi theta functions. In particular, all the mock theta functions encountered in these cases can be rewritten  in terms of the indefinite theta functions  \cite{zwegersthesis}, \cite{book}, \cite{gordon2012survey}, \cite{hickerson2014hecke}, with data satisfying the properties of Theorem \ref{ITtrace}. All the indefinite theta function have bilinear form $A=\left(\begin{smallmatrix}1 & 1\\ 1 & 0\end{smallmatrix}\right)$ and vectors $\bold{c}_1=\left(0,1\right), \bold{c}_2=(-1,1)$. The relations between mock theta functions and indefinite theta functions relevant for the cases considered are collected in appendix \ref{ITrep}. \\ In some cases it is not possible to directly specify  the individual Umbral McKay-Thompson series in terms of mock theta functions. When this happens, we will specify suitable linear combinations of the umbral McKay-Thompson series with disjoint sets of $q$-powers. In this way, the individual series can be retrieved by projecting onto the desired set of  $q$-powers. In fact, given an instance of umbral moonshine with Coxeter number $m$, the $r$-th component of the corresponding mock modular form will have a series expansion in which the appearing $q$-powers will have the general form $q^{-\frac{r^2}{4m}+N}$ with $N\in \mathbb{N}$. Thus, the $q$-series of components with different values of $r^2 \mod 4m$ will have no common $q$-powers and therefore a linear combination of such components contains the same information as the set of the individual components. \\
The expressions provided are obtained by making use of the explicit specification of some umbral McKay-Thompson series in terms of mock theta functions combined with the multiplicative relations among different lambencies, as provided in \cite{UM&Niemeier}. 
\subsection{Lambency Eight}
Lambency $\ell=8$ corresponds to the Niemeier root system $A^2_7D^2_5$ with umbral group $Dih_4$. The McKay-Thompson series appearing for $\ell=8$ can be expressed in terms of mock theta functions and eta quotients by making use of the multiplicative relations with $\ell=4$ and the explicit specifications in \cite{UM&Niemeier}. In particular, we encounter the order 2 mock theta functions 
\begin{align}
&A(q):=\sum\limits_{n=0}^{\infty}\frac{q^{n+1}(-q^2;q^2)_{n}}{(q; q^2)_{n+1}},\\
&B(q):=\sum \limits_{n=0}^{\infty}\frac{q^n(-q;q^2)_{n}}{(q;q^2)_{n+1}},
\end{align}
and the order 8 mock theta functions
\begin{align}
\label{S,T}
&S_0(q):=\sum\limits_{n=0}^{\infty}\frac{q^{n^2}(-q;q^2)_n}{(-q^2; q^2)_n},\\
&S_1(q):=\sum\limits_{n=0}^{\infty}\frac{q^{n(n+2)}(-q; q^2)_{n}}{(-q^2; q^2)},\\
&T_0(q):=\sum\limits_{n=0}^{\infty}\frac{q^{(n+1)(n+2)}(-q^2; q^2)_n}{(-q; q^2)_{n+1}},\\
&T_1(q):=\sum\limits_{n=0}^{\infty}q^{n(n+1)}\frac{(-q^2;q^2)_n}{(-q; q^2)_{n+1}}.
\end{align}
The expressions specifying all the components for all conjugacy classes of $Dih_4$ in terms of the previous functions are
\begin{equation}
\label{specl=8}
\begin{split}
&(H^{(8)}_{1A,1}-H^{(8)}_{1A,7})(2\tau)=H^{(4)}_{2C,1}(\tau)=q^{-\frac{1}{16}}(-2S_0(q)+4T_0(q)),\\
&H^{(8)}_{1A,2}(\tau)=H^{(8)}_{1A,6}(\tau)=4q^{-\frac{1}{4}}A(q),\\
&H^{(8)}_{1A,4}(\tau)=4q^{\frac{1}{2}}B(q),\\
&(H^{(8)}_{1A,3}-H^{(8)}_{1A,5})(2\tau)=H^{(4)}_{2C,3}(\tau)=q^{\frac{7}{16}}(2S_1(q)-4T_1(q)),\\
&(H^{(8)}_{2BC,1}-H^{(8)}_{2BC,7})(2\tau)=H^{(4)}_{4C,1}(\tau)=-2q^{-\frac{1}{16}}S_0(q)\\
&(H^{(8)}_{2BC,3}-H^{(8)}_{2BC,5})(2\tau)=H^{(4)}_{4C,3}(\tau)=2q^{\frac{7}{16}}S_1(q)\\
&(H^{(8)}_{4A,1}-H^{(8)}_{4A,7}-H^{(8)}_{4A,3}+H^{(8)}_{4A,5})(2\tau)=H^{(4)}_{4B,1}(\tau)-H^{(4)}_{4B,3}(\tau),\\
\end{split}
\end{equation}
together with the identities $H^8_{2BC,r}=H^8_{4A,r}=0$ for $r$ even, and the pairing relation
\begin{equation}
H^{(8)}_{2A,r}+(-1)^{r}H^{(8)}_{1A,r}=0.
\end{equation}
We can furthermore express the difference between components $r=1$ and $r=3$ for $\ell=4$ appearing in the relation for class 4A in terms of an eta quotient \footnote{This formula has a typo in the original paper.} \cite{UM&Niemeier}
\begin{equation}
H^{(4)}_{4B,1}(\tau)-H^{(4)}_{4B,3}(\tau)=-2\frac{\eta(\frac{\tau}{2})\eta(2\tau)^4}{\eta(\tau)^2\eta(4\tau)^2}.
\end{equation}
We can further simplify the previous expressions by making use of the following lemma  
\begin{lem}
The order 8 mock theta functions $S_0, S_1,T_0, T_1$ satisfy
\begin{equation}
\begin{split}
&S_0(q)+2T_0(q)=\frac{q^{\frac{1}{16}}}{2}\left(\frac{\eta(\frac{\tau}{2})^3}{\eta(\tau)\eta(2\tau)}+\frac{\eta(\tau)^8}{\eta(\frac{\tau}{2})^3\eta(2\tau)^4}\right),\\
&S_1(q)+2T_1(q)=\frac{q^{-\frac{7}{16}}}{2}\left(-\frac{\eta(\frac{\tau}{2})^3}{\eta(\tau)\eta(2\tau)}+\frac{\eta(\tau)^8}{\eta(\frac{\tau}{2})^3\eta(2\tau)^4}\right).
\end{split}
\end{equation}
\end{lem}
\begin{proof}
Using the expression in appendix A of \cite{book} 
\begin{equation}
\begin{split}
&S_0(q)=-2iq^{\frac{1}{2}}g_2(iq^{\frac{1}{2}};q^4)+\frac{(-iq^{\frac{1}{2}};-q)_{\infty}^2(-q;-q)_{\infty}(-q^3;q^8)_{\infty}(-q^5;q^8)_{\infty}}{(-q;q^4)_{\infty}(-q^3;q^4)_{\infty}(q^4;q^4)_{\infty}},\\
&S_1(q)=-2iq^{\frac{1}{2}}g_2(-iq^{\frac{3}{2}};q^4)+\frac{(-iq^{\frac{1}{2}};-q)_{\infty}^2(-q;-q)_{\infty}(-q;q^8)_{\infty}(-q^7;q^8)_{\infty}}{(-q;q^4)_{\infty}(-q^3;q^4)_{\infty}(q^4;q^4)_{\infty}}, \\
&\begin{split}
T_0(q)=&iq^{\frac{1}{2}}g_2(iq^{\frac{1}{2}};q^4)-\frac{(-iq^{\frac{1}{2}};-q)_{\infty}^2(-q;-q)_{\infty}(-q^3;q^8)_{\infty}(-q^5;q^8)_{\infty}}{2(-q;q^4)_{\infty}(-q^3;q^4)_{\infty}(q^4;q^4)_{\infty}}\\ &+\frac{1}{4}\frac{(q^{\frac{1}{2}};q^{\frac{1}{2}})^3_{\infty}}{(q;q)_{\infty}(q^2;q^2)_{\infty}}+\frac{1}{4}\frac{(q;q)^8_{\infty}}{(q^{\frac{1}{2}};q^{\frac{1}{2}})^3_{\infty}(q^2;q^2)^4_{\infty}},
\end{split}\\
&\begin{split}
T_1(q)=&iq^{\frac{1}{2}}g_2(-iq^{\frac{3}{2}};q^4)-\frac{(-iq^{\frac{1}{2}};-q)_{\infty}^2(-q;-q)_{\infty}(-q;q^8)_{\infty}(-q^7;q^8)_{\infty}}{2(-q;q^4)_{\infty}(-q^3;q^4)_{\infty}(q^4;q^4)_{\infty}} \\ &-\frac{q^{-\frac{1}{2}}}{4}\frac{(q^{\frac{1}{2}};q^{\frac{1}{2}})^3_{\infty}}{(q;q)_{\infty}(q^2;q^2)_{\infty}}+\frac{q^{-\frac{1}{2}}}{4}\frac{(q;q)^8_{\infty}}{(q^{\frac{1}{2}};q^{\frac{1}{2}})^3_{\infty}(q^2;q^2)^4_{\infty}},\\
\end{split}
\end{split}
\end{equation}
where $g_2$ is the universal mock theta function
\begin{equation}
g_2(\zeta; q)=\sum\limits_{n=0}^{\infty}\frac{(-q)_nq^{\frac{n(n+1)}{2}}}{(\zeta)_{n+1}(\zeta^{-1}q)_{n+1}},
\end{equation}
we can express $S_0$ ($S_1$ respectively) in terms of $T_0$ ($T_1$) and eta quotients. In fact, we can express the linear combinations $S_0+2T_0$, $S_1+2T_1$ as
\begin{equation}
\begin{split}
&S_0(q)+2T_0(q)=\frac{1}{2}\frac{(q^{\frac{1}{2}};q^{\frac{1}{2}})^3_{\infty}}{(q;q)_{\infty}(q^2;q^2)_{\infty}}+\frac{1}{2}\frac{(q;q)^8_{\infty}}{(q^{\frac{1}{2}};q^{\frac{1}{2}})^3_{\infty}(q^2;q^2)^4_{\infty}},\\
&S_1(q)+2T_1(q)=-\frac{q^{-\frac{1}{2}}}{2}\frac{(q^{\frac{1}{2}};q^{\frac{1}{2}})^3_{\infty}}{(q;q)_{\infty}(q^2;q^2)_{\infty}}+\frac{q^{-\frac{1}{2}}}{2}\frac{(q;q)^8_{\infty}}{(q^{\frac{1}{2}};q^{\frac{1}{2}})^3_{\infty}(q^2;q^2)^4_{\infty}},
\end{split}
\end{equation}
from which the conclusion since $\eta(\tau)=q^{\frac{1}{24}}(q;q)_{\infty}$.
\end{proof}
Using the previous relations we can rewrite the expressions for the components specifying  the Umbral McKay-Thompson series for all conjugacy class of the Umbral group $Dih_4$ as
\begin{equation}
\begin{split}
&(H^{(8)}_{1A,1}-H^{(8)}_{1A,7}-H^{(8)}_{1A,3}+H^{(8)}_{1A,5})(2\tau)=q^{-\frac{1}{16}}8T_0(q)+q^{\frac{7}{16}}8T_1(q)-2\frac{\eta(\tau)^8}{\eta(\frac{\tau}{2})^3\eta(2\tau)^4},\\
&(H^{(8)}_{2A,1}-H^{(8)}_{2A,7}-H^{(8)}_{2A,3}+H^{(8)}_{2A,5})(2\tau)=q^{-\frac{1}{16}}8T_0(q)+q^{\frac{7}{16}}8T_1(q)-2\frac{\eta(\tau)^8}{\eta(\frac{\tau}{2})^3\eta(2\tau)^4},\\
&(H^{(8)}_{2BC,1}-H^{(8)}_{2BC,7}-H^{(8)}_{2BC,3}+H^{(8)}_{2BC,5})(2\tau)=4q^{-\frac{1}{16}}T_0(q)+4q^{\frac{7}{16}}T_1(q)-2\frac{\eta(\tau)^8}{\eta(\frac{\tau}{2})^3\eta(2\tau)^4},\\
&(H^{(8)}_{4A,1}-H^{(8)}_{4A,7}-H^{(8)}_{4A,3}+H^{(8)}_{4A,5})(2\tau)=-2\frac{\eta(\frac{\tau}{2})\eta(\tau)^4}{\eta(\tau)^2\eta(4\tau)^2},\\
&H^{(8)}_{1A,2}(\tau)=H^{(8)}_{1A,6}(\tau)=4q^{-\frac{1}{8}}A(q),\\
&H^{(8)}_{1A,4}(\tau)=4q^{\frac{1}{2}}B(q).\\
\end{split}
\end{equation}
We can finally use the relations collected in appendix \ref{ITrep} to write all the appearing mock theta functions in terms of indefinite theta functions. 
\begin{prop} The expression specifying all the Mc-Kay Thompson series for $\ell=8$ at all conjugacy classes of the umbral group $Dih_4$ are
\begin{equation}
\begin{split}
&(H^{(8)}_{1A,1}-H^{(8)}_{1A,7}-H^{(8)}_{1A,3}+H^{(8)}_{1A,5})(2\tau)= \\ &8e^{-\frac{3 \pi i}{4}}\frac{\eta(4\tau)}{2\eta(2\tau)\eta(8\tau)}\left[\Theta_{\left(\frac{5}{8}, \frac{1}{8}\right), \left(\frac{1}{2},0\right)}(8\tau)-i\Theta_{\left(\frac{7}{8}, \frac{3}{8}\right), \left(\frac{1}{2},0\right)}(8\tau)\right]-2\frac{\eta(\tau)^8}{\eta(\frac{\tau}{2})^3\eta(2\tau)^4},\\
&(H^{(8)}_{2A,1}-H^{(8)}_{2A,7}-H^{(8)}_{2A,3}+H^{(8)}_{2A,5})(2\tau)=\\ &8e^{-\frac{3 \pi i}{4}}\frac{\eta(4\tau)}{2\eta(2\tau)\eta(8\tau)}\left[\Theta_{\left(\frac{5}{8}, \frac{1}{8}\right), \left(\frac{1}{2},0\right)}(8\tau)-i\Theta_{\left(\frac{7}{8}, \frac{3}{8}\right), \left(\frac{1}{2},0\right)}(8\tau)\right]-2\frac{\eta(\tau)^8}{\eta(\frac{\tau}{2})^3\eta(2\tau)^4},\\
&(H^{(8)}_{2BC,1}-H^{(8)}_{2BC,7}-H^{(8)}_{2BC,3}+H^{(8)}_{2BC,5})(2\tau)=\\ & 4e^{-\frac{3 \pi i}{4}}\frac{\eta(4\tau)}{2\eta(2\tau)\eta(8\tau)}\left[\Theta_{\left(\frac{5}{8}, \frac{1}{8}\right), \left(\frac{1}{2},0\right)}(8\tau)-i\Theta_{\left(\frac{7}{8}, \frac{3}{8}\right), \left(\frac{1}{2},0\right)}(8\tau)\right]-2\frac{\eta(\tau)^8}{\eta(\frac{\tau}{2})^3\eta(2\tau)^4},\\
&(H^{(8)}_{4A,1}-H^{(8)}_{4A,7}-H^{(8)}_{4A,3}+H^{(8)}_{4A,5})(2\tau)=-2\frac{\eta(\frac{\tau}{2})\eta(2\tau)^4}{\eta(\tau)^2\eta(4\tau)^2}, \\
& H^{(8)}_{1A,2}(\tau)=H^{(8)}_{1A,6}(\tau)=2e^{-\frac{3 \pi i}{4}}\frac{\eta(4\tau)}{\eta(2\tau)^2}\Theta_{\left(\frac{3}{4}, \frac{1}{4}\right), \left(0,\frac{1}{2}\right)}(4\tau),\\
&H^{(8)}_{1A,4}(\tau)=2e^{-\frac{3 \pi i}{4}}\frac{\eta(2\tau)}{\eta(\tau)\eta(4\tau)}\Theta_{\left(\frac{3}{4}, \frac{1}{2}\right), \left(0,\frac{1}{2}\right)}(4\tau).\\
\end{split}
\end{equation}
\end{prop}
We observe, in particular, that the indefinite theta functions appearing for the same components at different conjugacy classes are the same, thus the indefinite theta structure is invariant under the action of the umbral group. 
\subsection{Lambency Twelve}
At lambency $\ell=12$, we have Niemeier root system $A_{11}D_7E_6$ and umbral group $\mathbb{Z}_2$. The mock theta functions relevant in this case are the order 3
\begin{align}
&f(q):=\sum\limits_{n=0}^{\infty}\frac{q^{n^2}}{(-q;q)_n^2},\\
&\omega(q):=\sum\limits_{n=0}^{\infty}\frac{q^{2n(n+1)}}{(q; q^2)^2_{n+1}},
\end{align}
and the order 6 
\begin{align}
&\sigma(q):=\sum\limits_{n=0}^{\infty}\frac{q^{\frac{(n+1)(n+2)}{2}}(-q;q)_n}{(q; q^2)_{n+1}},\\
&\psi_6(q):=\sum\limits_{n=0}^{\infty}\frac{(-1)^nq^{(n+1)^2}(q;q^2)_n}{(-q;q)_{2n+1}}.
\end{align}
All the McKay-Thompson series for conjugacy class $2A$ are specified in terms of the ones for conjugacy class $1A$ by the pairing relation $H^{(12)}_{2A,r}+(-1)^{r}H^{(12)}_{1A,r}=0$.
In \cite{UM&Niemeier} we find the following identities in terms of mock theta functions 
\begin{equation}
\begin{split}
&H^{(12)}_{1A,2}(\tau)=H^{(12)}_{1A,10}(\tau)=-2q^{-\frac{4}{48}}\sigma(q),\\
&H^{(12)}_{1A,4}(\tau)=H^{(12)}_{1A,8}(\tau)=2q^{\frac{2}{3}}\omega(q).\\
\end{split}
\end{equation}
The multiplicative relations between $\ell=12$ and $\ell=6$
\begin{equation}
\begin{split}
&(H^{(12)}_{1A,1}-H^{(12)}_{1A,11})(2\tau)=H^{(6)}_{2B,1}(\tau),\\
&(H^{(12)}_{1A,5}-H^{(12)}_{1A,7})(2\tau)=H^{(6)}_{2B,5}(\tau),\\
&(H^{(12)}_{1A,3}-H^{(12)}_{1A,9})(2\tau)=H^{(6)}_{2B,3}(\tau),
\end{split}
\end{equation}
together with the multiplicative relations for $\ell=6$  
\begin{equation}
\begin{split}
&H^{(6)}_{2B,1}(3\tau)-H^{(6)}_{2B,3}(3\tau)+H^{(6)}_{2B,5}(3\tau)=H^{(2)}_{6A,1}(\tau),\\
&H^{(6)}_{2B,1}(2\tau)-H^{(6)}_{2B,5}(2\tau)=H^{(3)}_{4C}(\tau),
\end{split}
\end{equation}
and the following further explicit expressions in terms of mock theta functions
\begin{equation}
\begin{split}
&H^{(3)}_{4C,1}(\tau)=-2q^{-\frac{1}{12}}f(q^2),\\
&H^{(6)}_{2B,3}(\tau)=-2q^{-\frac{3}{8}}\psi_6(q),
\end{split}
\end{equation}
allow to specify all the components with $r$ odd in terms of mock theta functions and the function $H^{(2)}_{6A,1}$ as   
\begin{equation}
\begin{split}
&\left(H^{(12)}_{1A,1}-H^{(12)}_{1A,11}\right)(2\tau)=\frac{1}{2}H^{(2)}_{6A,1}\left(\frac{\tau}{3}\right)-q^{-\frac{3}{8}}\psi_6(q)-q^{-\frac{1}{24}}f(q),\\
&\left(H^{(12)}_{1A,5}-H^{(12)}_{1A,7}\right)(2\tau)=\frac{1}{2}H^{(2)}_{6A,1}\left(\frac{\tau}{3}\right)-q^{-\frac{3}{8}}\psi_6(q)+q^{-\frac{1}{24}}f(q),\\
&\left(H^{(12)}_{1A,3}-H^{(12)}_{1A,9}\right)(2\tau)=H^{(6)}_{2B,3}(\tau)=-2q^{-\frac{3}{8}}\psi_6(q).
\end{split}
\end{equation}
Finally, the multiplicative relations with $\ell=4$ give the component $r=6$
\begin{equation}
H^{(12)}_{1A,6}(3\tau)=H^{(12)}_{1A,2}(3\tau)+H^{(12)}_{1A,10}(3\tau)-H^{(4)}_{3A,2}(\tau)=-4q^{\frac{1}{4}}\sigma(q^3)-H^{(4)}_{3A,2}(\tau).
\end{equation}
For $H^{(4)}_{3A,2}$ and $H^{(2)}_{6A,1}$ a simple expression in terms of mock theta functions and/or eta quotients is not known, so we need to deal with them separately. It is convenient to write everything in terms of $\ell=4$ functions by using the multiplicative relation
\begin{equation}
\left(H^{(4)}_{3A,1}-H^{(4)}_{3A,3}\right)(2\tau)=H^{(2)}_{6A,1}(\tau).
\end{equation}
Components of the McKay-Thompson series at $\ell=4$ for conjugacy class $3A$ are specified by different powers of $y=e^{2\pi i z}$  in \cite{UMproof}
\begin{equation}
\label{eq:l=4}
2i \theta_1(3\tau, 6z)\theta_1(z, \tau)^{-1}\theta_1(3\tau, 3z)^{-1}\eta(\tau)^3=-2\mu^{0}_{4,0}(z, \tau)-2\mu^{1}_{4,0}(z, \tau)+\sum\limits_{r \mod 8}H^{(4)}_{3A,r}\theta_{4,r}(z, \tau)
\end{equation} 
Where we have made use of the following functions
\begin{equation}
\label{mixdefinitions}
\begin{split}
&\theta_{1}(z, \tau):=-iq^{\frac{1}{8}}y^{\frac{1}{2}}\prod\limits_{n>0}(1-y^{-1}q^{n-1})(1-yq^n)(1-q^n),\\
&\theta_2(z, \tau):=q^{\frac{1}{8}}y^{\frac{1}{2}}\prod\limits_{n>0}(1+y^{-1}q^{n-1})(1+yq^n)(1-q^n),\\
&\theta_{m,r}(z, \tau):=\sum\limits_{k\in \mathbb{Z}}y^{2mk+r}q^{\frac{(2mk+r)^2}{4m}},\\
&\mu^{k}_{m,0}(z, \tau):=\frac{1}{2}\left(\mu_{m,0}(z, \tau)+(-1)^k\mu_{m,0}\left(z, \tau+\frac{1}{2}\right)\right).
\end{split}
\end{equation}
We recall that the function $\mu_{m,0}(z, \tau)$, defined in \eqref{AL2}, has an expression in terms of indefinite theta functions. In fact, for  $\left |\frac{\text{Im}(z)}{\text{Im}(\tau)}\right|<1$, $\text{Im}(z)\neq 0$, setting $z=a\tau+b$ with $a\in \mathbb{Q}^{*}$, $|a|<1$, $b \in \mathbb{R}$ we can use the result in equation \eqref{specializedALIT} to write
\begin{equation}
\begin{split}
\label{l43A}
\sum\limits_{r \mod 8}H^{(4)}_{3A,r}(\tau)\theta_{4,r}(a\tau+b, \tau)=&-2\Theta^{+}_{A^{(4)}, \bold{c}^{(4)}_1, \bold{c}^{(4)}_2}(a\tau+b, 0;\tau)\\&+2i \theta_1(6a\tau+6b, 3\tau)\theta_1(a\tau+b, \tau)^{-1}\theta_1(3a\tau+3b,3\tau)^{-1}\eta(\tau)^3.
\end{split}
\end{equation}
with $A^{(m)}=\left(\begin{smallmatrix}2m & 1\\ 1 & 0\end{smallmatrix}\right)$, $\bold{c}_1^{(m)}=\left(0,1\right)$, $\bold{c}_2^{(m)}=\left(-1, 2m \right)$.
Notice also that equation \eqref{eq:l=4} implies that $H^{(4)}_{3A,r}$ have even coefficients. We can thus rewrite the umbral McKay Thompson series in terms of indefinite theta functions using the relations in appendix \ref{ITrep} as follows.
\begin{prop} The expression specifying all the McKay-Thompson series for $\ell=12$ at all conjugacy classes of the umbral group $\mathbb{Z}_2$ are given by 
\begin{equation}
\begin{split}
&\begin{split}
(H^{(12)}_{1A,1}-H^{(12)}_{1A,11})(2\tau)=&-e^{-\frac{7 \pi i}{6}}\frac{\eta(\tau)\eta(6\tau)}{2\eta(2\tau)\eta(3\tau)^2}\Theta_{\left(\frac{1}{3}, \frac{1}{2}\right), \left(\frac{1}{2}, \frac{1}{2}\right)}(3\tau)+\frac{2e^{-\frac{5\pi i}{6}}}{\eta(\tau)}\Theta_{\left(\frac{2}{3},\frac{1}{6}\right),\left(\frac{1}{2},0\right)}(3\tau)\\&-\frac{\eta(3\tau)^4}{\eta(\tau)\eta(6\tau)^2}+\frac{\left(H^{(4)}_{3A,1}-H^{(4)}_{3A,3}\right)}{2}\left(\frac{2}{3}\tau\right)
\end{split}\\
&\begin{split}
(H^{(12)}_{1A,5}-H^{(12)}_{1A,7})(2\tau)=&-e^{-\frac{7 \pi i}{6}}\frac{\eta(\tau)\eta(6\tau)}{2\eta(2\tau)\eta(3\tau)^2}\Theta_{\left(\frac{1}{3}, \frac{1}{2}\right), \left(\frac{1}{2}, \frac{1}{2}\right)}(3\tau)\\&-\frac{2e^{-\frac{5\pi i}{6}}}{\eta(\tau)}\Theta_{\left(\frac{2}{3},\frac{1}{6}\right),\left(\frac{1}{2},0\right)}(3\tau)+\frac{\eta(3\tau)^4}{\eta(\tau)\eta(6\tau)^2}+\frac{\left(H^{(4)}_{3A,1}-H^{(4)}_{3A,3}\right)}{2}\left(\frac{2}{3}\tau\right),
\end{split}\\
&(H^{(12)}_{1A,3}-H^{(12)}_{1A,9})(2\tau)=-e^{-\frac{7 \pi i}{6}}\frac{\eta(\tau)\eta(6\tau)}{\eta(2\tau)\eta(3\tau)^2}\Theta_{\left(\frac{1}{3}, \frac{1}{2}\right), \left(\frac{1}{2}, \frac{1}{2}\right)}\left(3\tau\right), \\
&H^{(12)}_{1A,2}(\tau)=H^{(12)}_{1A,10}(\tau)=-2 e^{-i\frac{\pi}{2}}\frac{\eta(2\tau)\eta(3\tau)}{2\eta(\tau)\eta(6\tau)^2}\Theta_{\left(\frac{1}{2}, \frac{1}{6}\right), \left(0, \frac{1}{2}\right)}(6\tau),\\
&H^{(12)}_{1A,4}(\tau)=H^{(12)}_{1A,8}(\tau)=\frac{2e^{-\frac{\pi i}{2}}}{\eta(\tau)}\Theta_{\left(\frac{1}{2}, \frac{1}{3}\right), \left(0, \frac{1}{2}\right)}(6\tau)+2\frac{\eta(6 \tau)^4}{\eta(2\tau)\eta(3\tau)^2},\\
&H^{(12)}_{1A,6}(3\tau)=-2e^{\frac{-\pi i}{2}}\frac{\eta(6\tau)\eta(9\tau)}{\eta(3\tau)\eta(18\tau)^2}\Theta_{\left(\frac{1}{2}, \frac{1}{6}\right), \left(0, \frac{1}{2}\right)}\left(18\tau\right) -H^{(4)}_{3A,2}(\tau).\\
\end{split}
\end{equation}
together with the pairing relation $H^{(12)}_{2A,r}+(-1)^{r}H^{(12)}_{1A,r}=0$.
\end{prop}
Again, we observe that in each component $H^{(12)}_{g,r}$ for given $r$, the indefinite theta part is invariant under the action of the umbral group, and the only difference between conjugacy class $1A$ and conjugacy class $2A$ is at most  an overall minus sign. 
\subsection{Lambency Sixteen}
At $\ell=16$ we have Niemeier root system $A_{15}D_{9}$ and umbral group $\mathbb{Z}_2$. Again, all the McKay-Thompson series for conjugacy class $2A$ 
are related to the one for class $1A$ by the pairing relation $H^{(16)}_{2A,r}+(-1)^{r}H^{(16)}_{1A,r}=0$.  As a result, we only need to specify $H^{(16)}_{1A,r}$ explicitly.  Using the expressions in \cite{UM&Niemeier} we can specify all the components of the Umbral McKay-Thompson series for class $1A$ in terms of order 8 mock thetas: $T_0(q)$ and $T_1(q)$ already defined in the previous section and 
\begin{align}
&U_0(q):=\sum\limits_{n=0}^{\infty}\frac{q^{n^2}(-q; q^2)_n}{(-q^4, q^4)_n},\\
&V_0(q):=-1+2\sum\limits_{n=0}^{\infty}\frac{q^{n^2}(-q;q^2)_n}{(q; q^2)_{n+1}},\\
&V_1(q):=\sum\limits_{n=0}^{\infty}\frac{q^{(n+1)^2}(-q; q^2)_n}{(q; q^2)_{n+1}},
\end{align}
 as 
\begin{equation}
\begin{split}
&H^{(16)}_{1A,2}(\tau)=H^{(16)}_{1A,14}(\tau)=2q^{-\frac{1}{16}}T_0(-q),\\
&H^{(16)}_{1A,4}(\tau)=H^{(16)}_{1A,12}(\tau)=2q^{-\frac{1}{4}}V_1(q),\\
&H^{(16)}_{1A,6}(\tau)=H^{(16)}_{1A,10}(\tau)=2q^{\frac{7}{16}}T_1(-q),\\
&H^{(16)}_{1A,8}(\tau)=V_0(q),\\
&\sum\limits_{n=0,7}(-1)^n H^{(16)}_{1A,2n+1}(8\tau)=H^{(2)}_{8A,1}(\tau)=-2q^{-\frac{1}{8}}U_0(q).
\end{split}
\end{equation}
Using the relations in appendix \ref{ITrep} we easily obtain 
\begin{prop} The expression specifying all the Mc-Kay Thompson series for $\ell=16$ at all conjugacy classes of the umbral group $\mathbb{Z}_2$ are
\begin{equation}
\begin{split}
&H^{(16)}_{1A,2}\left(\tau-\frac{1}{2}\right)=H^{(16)}_{1A,14}\left(\tau-\frac{1}{2}\right)=2e^{-\frac{3 \pi i}{4}}\frac{\eta(4\tau)}{2\eta(2\tau)\eta(8\tau)}\Theta_{\left(\frac{5}{8}, \frac{1}{8}\right), \left(\frac{1}{2},0\right)}(8\tau),\\
&H^{(16)}_{1A,4}(\tau)=H^{(16)}_{1A,12}(\tau)=2ie^{-\frac{3\pi i}{8}}\frac{q^{-\frac{1}{16}}}{2\theta_1(-\tau,8\tau)}\Theta_{\left(\frac{3}{8}, \frac{1}{4}\right), \left(0,\frac{1}{2}\right)}(8\tau),\\
&H^{(16)}_{1A,6}\left(\tau-\frac{1}{2}\right)=H^{(16)}_{1A,10}\left(\tau-\frac{1}{2}\right)=2e^{-\frac{5 \pi i}{4}}\frac{\eta(4\tau)}{2\eta(2\tau)\eta(8\tau)}\Theta_{\left(\frac{7}{8}, \frac{3}{8}\right), \left(\frac{1}{2},0\right)}(8\tau),\\
&H^{(16)}_{1A,8}(\tau)=-ie^{-\frac{\pi i}{8}}\frac{q^{-\frac{1}{16}}}{\theta_1(-\tau,8\tau)}\Theta_{\left(\frac{1}{8}, \frac{1}{2}\right), \left(0,\frac{1}{2}\right)}(8\tau)-\frac{\eta(2\tau)^3\eta(4\tau)}{\eta(\tau)^2\eta(8\tau)}, \\
&\sum\limits_{n=0,7}(-1)^n H^{(16)}_{1A,2n+1}(8\tau)=H^{(2)}_{8A,1}(\tau)=-2\frac{\eta(4\tau)}{2\eta(8\tau)^2}\Theta_{\left(\frac{1}{4}, \frac{1}{4}\right), (0,0)}(4\tau)
\end{split}
\end{equation}
together with the pairing relation $H^{(16)}_{2A,r}+(-1)^{r}H^{(16)}_{1A,r}=0$.
\end{prop}
We observe that also in this case the indefinite thetas appearing in all components are invariant under the action of the umbral group.
\begin{rmk}
	 The quantity $q^{\frac{1}{16}}\theta_1(-\tau, 8\tau)$  is modular under the congruence subgroup 
	 \begin{equation}
	 \Gamma_1(8):=\left\{\left(\begin{smallmatrix} a& b\\ c&d\end{smallmatrix}\right)\in \text{SL}(2, \mathbb{Z}): a,d = 1\bmod 8, c= 0\bmod 8\right\}
	 \end{equation}
 generated by the transformations $T: \tau \to \tau+1$, $\tilde{S}:\tau \to \frac{\tau}{8\tau+1}$, as is 
 %It is indeed 
 easy to see that from the transformation properties 
	\begin{equation}
	\begin{array}{ll}
	\theta_1(z,\tau+1)=e^{\frac{\pi i}{4}}\theta_1(z,\tau),
	&\theta_1\left(\frac{z}{\tau}, -\frac{1}{\tau}\right)=-i \sqrt{-i\tau}e^{\frac{i \pi z^2}{\tau}}\theta_1(z,\tau).
	\end{array}
	\end{equation}
%	one gets
%	\begin{equation}
%	\begin{split}
%	\theta_1\left(-\frac{\tau}{8\tau+1}, 8\frac{\tau}{8\tau+1}\right)=&
%	e^{\frac{i \pi}{4}}\theta_1\left(-\frac{\tau}{8\tau+1}, -\frac{1}{8\tau+1}\right)=-ie^{\frac{i \pi}{4}}\sqrt{-i \tau}e^{\frac{i \pi \tau^2}{8\tau+1}}\theta_1(-\tau,8\tau+1)\\ & -ie^{\frac{i \pi}{2}}\sqrt{-i \tau}e^{\frac{i \pi \tau^2}{8\tau+1}}\theta_1(-\tau,8\tau)=\sqrt{-i \tau}e^{\frac{i \pi \tau^2}{8\tau+1}}\theta_1(-\tau,8\tau)
%	\end{split}
%	\end{equation}
%	and then 
%	\begin{equation}
%	e^{\frac{2\pi i}{16} \frac{\tau}{8\tau+1}}\theta_1\left(-\frac{\tau}{8\tau+1}, 8\frac{\tau}{8\tau+1}\right)=e^{\frac{\pi i}{8} \frac{\tau}{8\tau+1}}e^{\frac{i \pi \tau^2}{8\tau+1}}\sqrt{-i \tau}\theta_1(-\tau,8\tau)=\sqrt{-i \tau}q^{\frac{1}{16}}\theta_1(-\tau, 8\tau).
%	\end{equation}
\end{rmk}
\section{Moonshine Modules}
In this section we will build modules whose trace functions reproduce the specifying expressions for the McKay-Thompson series provided in the previous section for lambency $\ell=8,12,16$. As mentioned in the previous section, in these cases we found that the the umbral groups act trivially on  all the indefinite theta functions appearing in the McKay-Thompson series. 
%is invariant under the action of the umbral group. 
Thus, we can construct modules that have the structure of a tensor product between an appropriate linear representation of the umbral group and a direct sum of vertex algebras modules on which the umbral group acts trivially. In the following, all the  trace functions defined as in \eqref{gradtrace} 
are trace functions of modules of subalgebras of 
 the vertex algebra associated to the two-dimensional  lattice with the indefinite quadratic form $A=\left(\begin{smallmatrix}1 &1 \\ 1 & 0\end{smallmatrix}\right)$.\\ 
We first start  by introducing some vertex algebra modules that will appear in our construction, and then provide explicit expressions for the relevant umbral moonshine modules. 
\subsection{Heisenberg, Clifford and Weyl Characters}
\label{SVOAmodules}
In this section we collect formulas for characters of (super) vertex algebras that will recover some of the functions appearing in the McKay-Thompson series specified in the previous section. Here we will follow the notation and definitions in \cite{moduleE8}, \cite{supervertexmeromorphic}, \cite{modulesDtype} for the super vertex operator algebras and their modules. \\
The simplest character we will need is the character of the Heisenberg vertex operator algebra $\mathcal{H}$ 
\begin{equation}
\chi^{H}(\tau):={\rm tr}_{\mathcal{H}}\left(q^{L(0)-\frac{c}{24}}\right)=\frac{1}{q^{\frac{1}{24}}\prod_{n>0}(1-q^n)}=\frac{1}{\eta(\tau)}.
\end{equation}
Next, we will consider the graded characters of the irreducible canonically-twisted modules of the Clifford vertex operator algebra $A^{\pm}_{\rm tw}$ \cite{moduleE8}
\begin{equation}
\chi^{A^{\pm}}(\tau):={\rm tr}_{A_{\rm tw}^{\pm}}\left(p(0)q^{L(0)-\frac{c}{24}}\right)=\pm q^{\frac{1}{24}}\prod\limits_{n>0}(1-q^n)=\pm \eta(\tau)
\end{equation}
as well as the character of the (d-dimensional) Clifford super vertex operator algebra canonically-twisted module $A_{\rm tw}$ \cite{supervertexmeromorphic}
\begin{equation}
\chi^{A_{\rm tw}}(z, \tau):={\rm tr}_{A_{\rm tw}}\left(y^{J(0)}q^{L(0)-\frac{d}{24}}\right)=y^{\frac{d}{4}}q^{\frac{d}{24}}\prod\limits_{n>0}(1+y^{-1}q^{n-1})^{\frac{d}{2}}(1+yq^{n})^{\frac{d}{2}}.
\end{equation}
Finally, we will also make use of the canonically twisted d-dimensional Weyl modules ${\rA}_{\rm tw}$ \cite{supervertexmeromorphic}
\begin{equation}
\chi^{\rA_{\rm tw}}(z, \tau):={\rm tr_{\rA_{\rm tw}}}(y^{J(0)}q^{L(0)-\frac{d}{24}})=y^{-\frac{d}{4}}q^{-\frac{d}{24}}\prod_{n>0}(1-y^{-1}q^{n-1})^{-\frac{d}{2}}(1-yq^n)^{-\frac{d}{2}}.
\end{equation}
\begin{rmk}
	The previous formula holds when each factor $(1-X)^{-1}$ is interpreted as $\sum\limits_{n \geq 0}X^{n}$, which is possible in the domain $0<-{\rm Im}(z)<{\rm Im}(\tau)$.
\end{rmk}
From now on, let's fix $d=2$ since this is the case that will be needed the following subsections. 
In particular, for $d=2$, we get the following relations with the Jacobi theta functions defined in \eqref{mixdefinitions}
\begin{equation}
\begin{split}
&\chi^{\rA_{\rm tw}}(z, \tau)=-i\frac{\eta(\tau)}{\theta_{1}(z,\tau)},\\
&\chi^{A_{\rm tw}}\left(z+\frac{1}{2},\tau\right)=-\frac{\theta_1\left(z, \tau\right)}{\eta(\tau)}.
\end{split}
\end{equation}
We will also need characters of 1-dimensional lattice vertex algebras. Let's consider the general 1-dimensional (even) lattice $L^1:=\{\alpha\epsilon: \alpha \in \mathbb{Z}\}$ with scalar product $\braket{\epsilon, \epsilon}=2m$. Let's recall the operator $g_h$ for $h:=\epsilon \otimes h \in L^1\otimes_\ZZ \mathbb{Q}$ defined in \eqref{auth}. We have 
\begin{equation}
\chi^{L^1}_{h}(\tau):=\text{Tr}_{V_{L^1}}(g_hq^{L_0-\frac{c}{24}})=\frac{1}{\eta(\tau)}\sum\limits_{n \in \mathbb{Z}}e^{4\pi i m h n}q^{mn^2}.
\end{equation}
Furthermore, the characters of the modules $V_{L^1+\frac{r}{2m}}$, for $0<r<2m$
\begin{equation}
\chi^{L^1+\frac{r}{2m}}_{h}(\tau):=\text{Tr}_{V_{L^1+\frac{r}{2m}}}(g_{h}q^{L_0-\frac{c}{24}})=\frac{1}{\eta(\tau)}\sum\limits_{n \in \mathbb{Z}}e^{2\pi i h (2mn+r)}q^{\frac{(2mn+r)^2}{4m}}.
\end{equation}
give the theta functions $\theta_{m,r}$ defined in \eqref{mixdefinitions}.
Since they will  appear frequently later, let's give special names to the following characters of the vertex algebra $V_{L^1}$ associated to the 1 dimensional lattice $L^1:=\{\alpha\epsilon: \alpha \in \mathbb{Z}\}$  with scalar product $\braket{\epsilon, \epsilon}=2$, and the vertex algebra $V_K$ associated to the sublattice $K\subset L^1=\{\alpha \epsilon: \alpha \in \mathbb{Z}_{\geq 0}\}$. Introducing the operator
\begin{equation}
\label{1d-grad}
g_{\frac{1}{4}}(p\otimes n\epsilon)=(-1)^n (p\otimes n\epsilon)
\end{equation}
which corresponds to \eqref{auth} with the choice $h=\frac{1}{4}\epsilon$, we define 
\begin{equation}
\label{1Dchar}
\begin{split}
&\chi^{L^1}(\tau):= {\rm Tr}_{V_{L^1}}\left(q^{L_0-\frac{c}{24}}\right)=\frac{1}{\eta(\tau)}\sum\limits_{n \in \mathbb{Z}}q^{n^2},\\
&\chi^{K}(\tau):= {\rm Tr}_{V_K}\left(q^{L_0-\frac{c}{24}}\right)=\frac{1}{\eta(\tau)}\sum\limits_{n\geq 0}q^{n^2},\\
&\tilde{\chi}^{K}(\tau):={\rm Tr}_{V_K}\left(g_{\frac{1}{4}} q^{L_0-\frac{c}{24}}\right)=\frac{1}{\eta(\tau)}\sum\limits_{n\geq 0}(-1)^nq^{n^2}.\\
\end{split}
\end{equation}
\subsection{Lambency Eight}
The umbral group for 
lambency $\ell=8$ is $G=Dih_4$. We will use the conventions for the names of conjugacy classes and irreducible representations that are specified in the character table \ref{CTD4}.
\begin{table}
	\caption{Character table of $Dih_4$}
	\label{CTD4}
\begin{center}
	\begin{tabular}{c|ccccc}
		& 1A & 2A & 2B & 2C & 4A \\
		\hline
		$A_1$	& 1 & 1 & 1 & 1& 1\\
		$A_2$	& 1 & 1 & -1 & -1 & 1 \\
		$B_1$	& 1 & 1 & -1 & 1 & -1\\
		$B_2$	& 1 & 1 & 1 & -1 & -1\\
		$E$		& 2 & -2 & 0 & 0 & 0

	\end{tabular}
\end{center}
\end{table}
Using the results of the previous sections, we can specify the McKay-Thompson series for $\ell=8$ in terms of characters of the VOAs introduced before. The even components can be directly rewritten as
\begin{equation}
\begin{split}
&H^{(8)}_{g,2}(\tau)=H^{(8)}_{g,6}(\tau)=2{\rm tr}_{E_2}(g)\chi^{A^+}(4\tau)\chi^{A^+}(\tau)^2\chi^{H}(2\tau)^2T^{(4)}_{\left(\frac{3}{4}, \frac{1}{4}\right),\left(0,\frac{1}{2}\right)}(\tau), \\
&H^{(8)}_{g,4}(\tau)=2{\rm tr}_{E_2}(g)\chi^{A^+}(2\tau)\chi^{H}(\tau)\chi^{H}(4\tau)\chi^{A^+}(\tau)^2T^{(4)}_{\left(\frac{3}{4}, \frac{1}{2}\right),\left(0,\frac{1}{2}\right)}(\tau),
\end{split}
\end{equation}
while the odd components are specified by 
\begin{equation}
\begin{split}
&(H^{(8)}_{g,1}-H^{(8)}_{g,7}-H^{(8)}_{g,3}+H^{(8)}_{g,5})(2\tau)=\\
& 2{\rm tr}_{2A_1\oplus B_1\oplus B_2}(g)\chi^{A^{+}}(4\tau)\chi^{H}(2\tau)\chi^H(8\tau)\chi^{A^{+}}(\tau)\left[\chi^{A^+}(\tau)T^{(8)}_{\left(\frac{5}{8}, \frac{1}{8}\right),\left(\frac{1}{2},0\right)}(\tau)+\chi^{A^{-}}(\tau)T^{(8)}_{\left(\frac{7}{8}, \frac{3}{8}\right),\left(\frac{1}{2},0\right)}(\tau)\right]\\
&+\left[\chi^{A^{+}}\left(\frac{\tau}{2}\right)tr_{A_1\oplus A_2} (g) +\chi^{A^{-}}\left(\frac{\tau}{2}\right) tr_{B_1\oplus B_2} (g)\right]\Biggl[\tilde{\chi}^K(\tau)\chi^{L^1}(\tau)+\\ &\chi^{A^{-}}\left(\frac{\tau}{2}\right)\chi^{A^{+}}\left(\frac{\tau}{2}\right) \chi^{\mathcal{H}}(\tau)^2\chi^K\left(\frac{\tau}{2}\right)\chi^{L^1}\left(\frac{\tau}{2}\right)+\chi^{\mathcal{H}}(\tau)\chi^K(\tau)+ \chi^{A^{+}}\left(\frac{\tau}{2}\right)\chi^{\mathcal{H}}(\tau)^2\chi^K\left(\frac{\tau}{2}\right)\Biggr]\\ &+2{\rm tr}_{A_1}(g)\chi^{A^{-}}(\tau)\chi^{A^{+}}(\tau)^7\chi^{\mathcal{H}}\left(\frac{\tau}{2}\right)^3\chi^{H}(2\tau)^4.
\end{split}
\end{equation}
In rewriting the second addend we have used the following lemma so that the prefactor multiplying the characters is integer
\begin{lem}
\begin{equation}
\begin{split}
\frac{\eta(\frac{\tau}{2})\eta(2\tau)^4}{\eta(\tau)^2\eta(4\tau)^2}-\frac{\eta(\tau)^8}{\eta(\frac{\tau}{2})^3\eta(2\tau)^4}=2\chi^{A^{+}}\left(\frac{\tau}{2}\right)\Biggl[& \tilde{\chi}^K(\tau)\chi^{L^1}(\tau)-\chi^{A^{+}}\left(\frac{\tau}{2}\right)^2\chi^{\mathcal{H}}(\tau)^2\chi^K\left(\frac{\tau}{2}\right)\chi^{L^1}\left(\frac{\tau}{2}\right)\\ &+\chi^{\mathcal{H}}(\tau)\chi^K(\tau)+ \chi^{A^{+}}\left(\frac{\tau}{2}\right)\chi^{\mathcal{H}}(\tau)^2\chi^K\left(\frac{\tau}{2}\right)\Biggr]
\end{split}
\end{equation}
\end{lem}
\begin{proof}
Using the identities \cite{eta&theta}
\begin{equation}
\begin{array}{ll}
\label{identities}
\frac{\eta(2\tau)^5}{\eta(\tau)^2\eta(4\tau)^2}=\sum\limits_{n \in \mathbb{Z}}q^{n^2}=:\theta^1(\tau), &
\frac{\eta(\tau)^2}{\eta(2\tau)}=\sum\limits_{n \in \mathbb{Z}}(-1)^n q^{n^2},
\end{array}
\end{equation}
we get
\begin{equation}
\begin{split}
&\frac{\eta(\frac{\tau}{2})\eta(2\tau)^4}{\eta(\tau)^2\eta(4\tau)^2}-\frac{\eta(\tau)^8}{\eta(\frac{\tau}{2})^3\eta(2\tau)^4}= \frac{\eta\left(\frac{\tau}{2}\right)}{\eta(2\tau)}\frac{\eta(2\tau)^5}{\eta(\tau)^2\eta(4\tau)^2}-\frac{\eta\left(\frac{\tau}{2}\right)}{\eta(\tau)^2}\left(\frac{\eta(\tau)^5}{\eta\left(\frac{\tau}{2}\right)^2\eta(2\tau)^2}\right)^2\\ 
& =\frac{\eta\left(\frac{\tau}{2}\right)}{\eta(\tau)^2}\left[\frac{\eta(\tau)^2}{\eta(2\tau)} \theta_1(\tau)-\theta_1\left(\frac{\tau}{2}\right)^2\right]=
\frac{\eta\left(\frac{\tau}{2}\right)}{\eta(\tau)^2}\left[\sum\limits_{m,n\in \mathbb{Z}}(-1)^n q^{m^2+n^2}-q^{\frac{m^2+n^2}{2}}\right]\\&
=2\frac{\eta(\frac{\tau}{2})}{\eta(\tau)^2}\left[
\sum_{\substack{n,m \in \mathbb{Z}\\ n\geq 0}}(-1)^n q^{m^2+n^2}-\sum_{\substack{n,m \in \mathbb{Z}\\ n \geq 0}}q^{\frac{m^2+n^2}{2}}
-\sum\limits_{n\geq 0}q^{n^2}+\sum\limits_{n\geq0}q^{\frac{n^2}{2}}
\right]
\end{split}
\end{equation}
and the conclusion follows easily using the expressions for the characters provided in \eqref{1Dchar}.
\end{proof}
In order to specify the trace functions that will give us the relevant umbral McKay-Thompson series at $\ell=8$, let's define the modules 
\begin{align*}
\frak{M}_{1,1}^{(8)}&:={A^{+}_{tw}}^{\otimes 3}\otimes\mathcal{H}^{\otimes 2}\otimes V^{(8)}_{\left(\frac{5}{8}, \frac{1}{8}\right)},\\
\frak{M}_{1,2}^{(8)}&:={A^{+}_{tw}}^{\otimes 2}\otimes A^{-}_{tw}\otimes\mathcal{H}^{\otimes 2}\otimes V^{(8)}_{\left(\frac{7}{8}, \frac{3}{8}\right)},\\
\frak{M}_{1,3}^{(8)}&:={A^{+}_{tw}}\otimes K\otimes L^1,\\
\frak{M}_{1,4}^{(8)}&:={A^{+}_{tw}}^{\otimes 2}\otimes A_{tw}^{-}\otimes \mathcal{H}^{\otimes 2}\otimes K \otimes L^1,\\
\frak{M}^{(8)}_{1,5}&:=A^{+}_{tw}\otimes\mathcal{H}\otimes K,\\
\frak{M}^{(8)}_{1,6}&:={A^{+}_{tw}}^{\otimes 2}\otimes \mathcal{H}^{\otimes 2} \otimes K,\\
\frak{M}_{1,7}^{(8)}&:={A^{-}_{tw}}\otimes K\otimes L^1,\\
\frak{M}_{1,8}^{(8)}&:={A^{-}_{tw}}^{\otimes 2}A^{+}_{tw}\otimes \mathcal{H}^{\otimes 2}\otimes K \otimes L^1\\
\frak{M}^{(8)}_{1,9}&:=A^{-}_{tw}\otimes\mathcal{H}\otimes K,\\
\frak{M}^{(8)}_{1,10}&:=A^{+}_{tw}\otimes A^{-}_{tw}\otimes \mathcal{H}^{\otimes 2} \otimes K,\\
\frak{M}^{(8)}_{1,11}&:=A^{-}_{tw}\otimes{A^{+}_{tw}}^{\otimes 7}\otimes\mathcal{H}^{\otimes 7},\\
\frak{M}^{(8)}_2&:={A^{+}_{tw}}^{\otimes 3}\otimes\mathcal{H}^{\otimes 2}\otimes V^{(4)}_{\left(\frac{3}{4}, \frac{1}{4}\right)},\\
\frak{M}^{(8)}_4&:={A^{+}_{tw}}^{\otimes 3}\otimes\mathcal{H}^{\otimes 2}\otimes V^{(4)}_{\left(\frac{3}{4}, \frac{1}{2}\right)},
\end{align*}
and for each of them let's define the vectors 
\begin{align*}
\omega_{1,1}^{(8)}&:=2\hat{\omega}^{(1)}+\frac{1}{2}\hat{\omega}^{(2)}+\frac{1}{2}\hat{\omega}^{(3)}+\hat{\omega}^{(4)}+4\hat{\omega}^{(5)}+\frac{1}{2}\hat{\omega}^{(6)},\\
\omega_{1,2}^{(8)}&:=2\hat{\omega}^{(1)}+\frac{1}{2}\hat{\omega}^{(2)}+\frac{1}{2}\hat{\omega}^{(3)}+\hat{\omega}^{(4)}+4\hat{\omega}^{(5)}+\frac{1}{2}\hat{\omega}^{(6)},\\
\omega^{(8)}_{1,3}&:=\frac{1}{4}\hat{\omega}^{(1)}+\frac{1}{2}\hat{\omega}^{(2)}+\frac{1}{2}\hat{\omega}^{(3)},\\
\omega^{(8)}_{1,4}&:=\frac{1}{4}\hat{\omega}^{(1)}+\frac{1}{4}\hat{\omega}^{(2)}+\frac{1}{4}\hat{\omega}^{(3)}+\frac{1}{2}\hat{\omega}^{(4)}+\frac{1}{2}\hat{\omega}^{(5)}+\frac{1}{4}\hat{\omega}^{(6)}+\frac{1}{4}\hat{\omega}^{(7)},\\
\omega^{(8)}_{1,5}&:=\frac{1}{4}\hat{\omega}^{(1)}+\frac{1}{4}\hat{\omega}^{(2)}+\frac{1}{4}\hat{\omega}^{(3)},\\
\omega^{(8)}_{1,6}&:=\frac{1}{4}\hat{\omega}^{(1)}+\frac{1}{4}\hat{\omega}^{(2)}+\frac{1}{2}\hat{\omega}^{(3)}+\frac{1}{2}\hat{\omega}^{(4)}+\frac{1}{4}\hat{\omega}^{(5)},\\
\omega^{(8)}_{1,7}&:=\frac{1}{4}\hat{\omega}^{(1)}+\frac{1}{2}\hat{\omega}^{(2)}+\frac{1}{2}\hat{\omega}^{(3)},\\
\omega^{(8)}_{1,8}&:=\frac{1}{4}\hat{\omega}^{(1)}+\frac{1}{4}\hat{\omega}^{(2)}+\frac{1}{4}\hat{\omega}^{(3)}+\frac{1}{2}\hat{\omega}^{(4)}+\frac{1}{2}\hat{\omega}^{(5)}+\frac{1}{4}\hat{\omega}^{(6)}+\frac{1}{4}\hat{\omega}^{(7)},\\
\omega^{(8)}_{1,9}&:=\frac{1}{4}\hat{\omega}^{(1)}+\frac{1}{2}\hat{\omega}^{(2)}+\frac{1}{2}\hat{\omega}^{(3)},\\
\omega^{(8)}_{1,10}&:=\frac{1}{4}\hat{\omega}^{(1)}+\frac{1}{4}\hat{\omega}^{(2)}+\frac{1}{2}\hat{\omega}^{(3)}+\frac{1}{2}\hat{\omega}^{(4)}+\frac{1}{4}\hat{\omega}^{(5)},\\
\begin{split}
\omega^{(8)}_{1,11}&:=\frac{1}{2}\hat{\omega}^{(1)}+\frac{1}{2}\hat{\omega}^{(2)}+\frac{1}{2}\hat{\omega}^{(3)}+\frac{1}{2}\hat{\omega}^{(4)}+\frac{1}{2}\hat{\omega}^{(5)}+\frac{1}{2}\hat{\omega}^{(6)}+\frac{1}{2}\hat{\omega}^{(7)}+\frac{1}{2}\hat{\omega}^{(8)}+\frac{1}{4}\hat{\omega}^{(9)}+\frac{1}{4}\hat{\omega}^{(10)}+\frac{1}{4}\hat{\omega}^{(11)}\\ &+\hat{\omega}^{(12)}+\hat{\omega}^{(13)}+\hat{\omega}^{(14)}+\hat{\omega}^{(15)},
\end{split}\\
\omega^{(8)}_{2}&:=4\hat{\omega}^{(1)}+\hat{\omega}^{(2)}+\hat{\omega}^{(3)}+2\hat{\omega}^{(4)}+2\hat{\omega}^{(5)}+\hat{\omega}^{(6)},\\
\omega^{(8)}_{4}&:=2\hat{\omega}^{(1)}+\hat{\omega}^{(2)}+\hat{\omega}^{(3)}+\hat{\omega}^{(4)}+4\hat{\omega}^{(5)}+\hat{\omega}^{(6)},
\end{align*}
where, for brevity, we have written $\hat{\omega}^{(i)}=\bold{v}\otimes \cdots \otimes \left(\omega^{(i)}-\frac{c^{(i)}}{24}\mathbf{v}\right) \otimes \cdots \otimes \bold{v}$ to indicate the tensor product of vectors that at position $i$ has the factor $\omega^{(i)}-\frac{c^{(i)}}{24}\mathbf{v}$, where $\omega$ and $c$ are respectively the conformal vector and central charge of the module at the $i$-th position, and the remaining factors are the vacuum vectors $\bold{v}$ of the other modules. Let's consider the operators\footnote{To make the notation lighter we will not write the indices in $\hat{L}$. It is understood that, for each module, $\hat{L}$ corresponds to the vector associated to the module.} $\hat{L}(0)$ corresponding to the 0-modes of the vertex operators associated to the previous vectors. With this notation we get
\begin{thm}
\label{l8th}
The umbral McKay-Thompson series at lambency $\ell=8$ are specified by
\begin{equation}
\begin{split}
&H^{(8)}_{g,2}(\tau)=H^{(8)}_{g,6}(\tau)=2{\rm tr}_{E_2}(g){\rm tr}_{\frak{M}^{(8)}_2}\left(g_{\left(0, \frac{1}{2}\right)}q^{\hat{L}(0)}\right),\\
&H^{(8)}_{g,4}(\tau)=2{\rm tr}_{E_2}(g){\rm tr}_{\frak{M}^{(8)}_4}\left(g_{\left(0, \frac{1}{2}\right)}q^{\hat{L}(0)}\right),\\
&\begin{split}
&(H^{(8)}_{g,1}-H^{(8)}_{g,7}-H^{(8)}_{g,3}+H^{(8)}_{g,5})(\tau)=
2{\rm tr}_{2A_1\oplus B_1\oplus B_2}(g){\rm tr}_{\frak{M}^{(8)}_{1,1}\oplus\frak{M}^{(8)}_{1,2}}\left(g_{\left(\frac{1}{2}, 0\right)}q^{\hat{L}(0)}\right)\\
&+{\rm tr}_{A_1\oplus A_2} (g) {\rm tr}_{\frak{M}^{(8)}_{1,3}}\left(g_{\frac{1}{4}} q^{\hat{L}(0)}\right)+{\rm tr}_{B_1\oplus B_2}(g){\rm tr}_{\frak{M}^{(8)}_{1,7}}\left(g_{\frac{1}{4}} q^{\hat{L}(0)}\right) \\
&+{\rm tr}_{A_1\oplus A_2} (g) {\rm tr}_{\frak{M}^{(8)}_{1,4}\oplus \frak{M}^{(8)}_{1,5}\oplus \frak{M}^{(8)}_{1,6}}\left( q^{\hat{L}(0)}\right)+ {\rm tr}_{B_1\oplus B_2}(g){\rm tr}_{\frak{M}^{(8)}_{1,8}\oplus \frak{M}^{(8)}_{1,9}\oplus \frak{M}^{(8)}_{1,10}}\left(q^{\hat{L}(0)}\right)\\ &+2{\rm tr}_{A_1}(g){\rm tr}_{\frak{M}^{(8)}_{1,11}}\left( q^{\hat{L}(0)}\right),
\end{split}
\end{split}
\end{equation}
where $g_{\bold{b}}$ acts as specified in \eqref{twisting} on the cone vertex algebra module in the tensor product and trivially on all the others. Analogously $g_{\frac{1}{4}}$ is specified by \eqref{1d-grad} and only acts non-trivially on the module $K$.
\end{thm}
\subsection{Lambency Twelve}
The umbral group corresponding to $\ell=12$ is $\mathbb{Z}/2\mathbb{Z}$. There are only 2 irreducible representations, we will call $A$ the trivial representation and $B$ the sign representation. \\
We can specify the McKay-Thompson series in terms of characters of vertex algebras and $H^{(4)}$ functions. Let's write 
\begin{align}
\label{e}
&e_4(\tau)=-H^{(4)}_{3A,2}(\tau)\\
\label{o}
&o_4(\tau)=\left(\frac{H^{(4)}_{3A,1}- H^{(4)}_{3A,3}}{2}\right)\left(\frac{2}{3}\tau\right)
\end{align}
The odd components are specified by
\begin{equation}
\label{l12odd}
\begin{split}
&\left(H^{(12)}_{g,1}-H^{(12)}_{g,11}\right)(2\tau)=\\
&\begin{split} {\rm tr}_{A}(g)\Biggl[&\chi^{A^{-}}(\tau)\chi^{A^{+}}\left(\frac{\tau}{2}\right)^2\chi^{A^{+}}(6\tau)\chi^{\mathcal{H}}(2\tau)\chi^{\mathcal{H}}(3\tau)T^{(6)}_{\left(\frac{1}{3}, \frac{1}{2}\right), \left(\frac{1}{2}, \frac{1}{2}\right)}\left(\frac{\tau}{2}\right)\\&+4\chi^{A^{+}}\left(\frac{\tau}{2}\right)^2\chi^{\mathcal{H}}(\tau)T^{(6)}_{\left(\frac{2}{3}, \frac{1}{6}\right), \left(\frac{1}{2},0\right)}\left(\frac{\tau}{2}\right)\\
&+\chi^{A^{+}}(3\tau)^2\chi^{A^{-}}(3\tau)\chi^{\mathcal{H}}(\tau)\chi^{\mathcal{H}}(6\tau)^2+o_4(\tau)\Biggr],
\end{split}
\\
&\left(H^{(12)}_{g,5}-H^{(12)}_{g,7}\right)(2\tau)=\\
&\begin{split} {\rm tr}_{A}(g)\Biggl[&\chi^{A^{-}}(\tau)\chi^{A^{+}}\left(\frac{\tau}{2}\right)^2\chi^{A^{+}}(6\tau)\chi^{\mathcal{H}}(2\tau)\chi^{\mathcal{H}}(3\tau)T^{(6)}_{\left(\frac{1}{3}, \frac{1}{2}\right), \left(\frac{1}{2}, \frac{1}{2}\right)}\left(\frac{\tau}{2}\right)\\ &+4\chi^{A^{-}}\left(\frac{\tau}{2}\right)\chi^{A^{+}}\left(\frac{\tau}{2}\right)\chi^{\mathcal{H}}(\tau)T^{(6)}_{\left(\frac{2}{3}, \frac{1}{6}\right), \left(\frac{1}{2},0\right)}\left(\frac{\tau}{2}\right)+\chi^{A^{+}}(3\tau)^2\chi^{A^{+}}(3\tau)\chi^{\mathcal{H}}(\tau)\chi^{\mathcal{H}}(6\tau)^2\\ &+o_4(\tau)\Biggr],
\end{split}
\\
&\left(H^{(12)}_{g,3}-H^{(12)}_{g,9}\right)(2\tau)=2{\rm tr}_{A}(g)\chi^{A^{-}}(\tau)\chi^{A^{+}}\left(\frac{\tau}{2}\right)\chi^{A^{+}}(6\tau)\chi^{\mathcal{H}}(2\tau)\chi^{\mathcal{H}}(3\tau)^2T^{(6)}_{\left(\frac{1}{3}, \frac{1}{2}\right),\left(\frac{1}{2},\frac{1}{2}\right)}\left(\frac{\tau}{2}\right),\\
\end{split}
\end{equation}
The even components are instead given by 
\begin{equation}
\label{l12even}
\begin{split}
&H^{(12)}_{g,2}(\tau)=H^{(12)}_{g,10}(\tau)=2tr_{B}(g)\chi^{\mathcal{H}}(6\tau)^2\chi^{A^{-}}(\tau)\chi^{A^{+}}(2\tau)\chi^{A^{+}}(3\tau)T^{(6)}_{\left(\frac{1}{2}, \frac{1}{6}\right), \left(0, \frac{1}{2}\right)}(\tau),\\
&H^{(12)}_{g,4}(\tau)=H^{(12)}_{g,8}(\tau)=4{\rm tr}_{B}(g)\left[\chi^{A^{+}}(\tau)T^{(6)}_{\left(\frac{1}{2}, \frac{1}{3}\right), \left(0,\frac{1}{2}\right)}(\tau)+2\chi^{A^{+}}(6\tau)^4\chi^{\mathcal{H}}(2\tau)\chi^{\mathcal{H}}(3\tau)^2\right],\\ 
&H^{(12)}_{g, 6}(3\tau)={\rm tr}_{B}(g)\left[4\chi^{A^{-}}(6\tau)\chi^{A^{+}}(9\tau)\chi^{A^{+}}(\tau)^2\chi^{\mathcal{H}}(3\tau)\chi^{\mathcal{H}}(18\tau)^2T^{(18)}_{\left(\frac{1}{2}, \frac{1}{6}\right), \left(0, \frac{1}{2}\right)}(\tau)+e_4(\tau)\right].
\end{split}
\end{equation}
We define the modules
\begin{align*}
&\frak{M}^{(12)}_{1,1}:=A^{-}_{tw}\otimes {A^{+}_{tw}}^{\otimes 3}\otimes \mathcal{H}^{\otimes 2}\otimes V^{(6)}_{\left(\frac{1}{3}, \frac{1}{2}\right)},\\
&\frak{M}^{(12)}_{1,2}:={A^{+}_{tw}}^{\otimes 2}\otimes \mathcal{H}\otimes V^{(6)}_{\left(\frac{2}{3}, \frac{1}{6}\right)},\\
&\frak{M}^{(12)}_{1,3}:={A^{+}_{tw}}^{\otimes 2}\otimes A^{-}_{tw} \otimes \mathcal{H}^{\otimes 3},\\
&\frak{M}^{(12)}_{2}:=A^{-}_{tw}\otimes {A^{+}_{tw}}^{\otimes 2}\otimes \mathcal{H}^{\otimes 2}\otimes V^{(6)}_{\left(\frac{1}{2}, \frac{1}{6}\right)},\\
&\frak{M}^{(12)}_{3}:=A^{-}_{tw}\otimes {A^{+}_{tw}}^{\otimes 2}\otimes \mathcal{H}^{\otimes 3}\otimes V^{(6)}_{\left(\frac{1}{3}, \frac{1}{2}\right)},\\
&\frak{M}^{(12)}_{4,1}:=A^{+}_{tw}\otimes V^{(6)}_{\left(\frac{1}{2}, \frac{1}{3}\right)},\\
&\frak{M}^{(12)}_{4,2}:={A^{+}_{tw}}^{\otimes 4}\otimes \mathcal{H}^{\otimes 3}\\
&\frak{M}^{(12)}_{5,1}:=A^{-}_{tw}\otimes {A^{+}_{tw}}^{\otimes 3}\otimes \mathcal{H}^{\otimes 2}\otimes V^{(6)}_{\left(\frac{1}{3}, \frac{1}{2}\right)},\\
&\frak{M}^{(12)}_{5,2}:=A^{-}_{tw}\otimes A^{+}_{tw}\otimes \mathcal{H}\otimes V^{(6)}_{\left(\frac{2}{3}, \frac{1}{6}\right)}\\
&\frak{M}^{(12)}_{5,3}:={A^{+}_{tw}}^{\otimes 3}\otimes \mathcal{H}^{\otimes 3}\\
&\frak{M}^{(12)}_{6}:=A^{-}_{tw}\otimes {A^{+}_{tw}}^{\otimes 3}\otimes \mathcal{H}^{\otimes 3}\otimes V^{(18)}_{\left(\frac{1}{2}, \frac{1}{6}\right)},\\
\end{align*}
and, to account for the different coefficients in front of $\tau$, the vectors  
\begin{align*}
&\omega_{1,1}^{(12)}:=\frac{1}{2}\hat{\omega}^{(1)}+\frac{1}{4}\hat{\omega}^{(2)}+\frac{1}{2}\hat{\omega}^{(3)}+3\hat{\omega}^{(4)}+\hat{\omega}^{(5)}+\frac{3}{2}\hat{\omega}^{(6)}+\frac{1}{4}\hat{\omega}^{(7)},\\
&\omega_{1,2}^{(12)}:=\frac{1}{4}\hat{\omega}^{(1)}+\frac{1}{4}\hat{\omega}^{(2)}+\frac{1}{2}\hat{\omega}^{(3)}+\frac{1}{4}\hat{\omega}^{(4)},\\
&\omega_{1,3}^{(12)}:=\frac{3}{2}\hat{\omega}^{(1)}+\frac{3}{2}\hat{\omega}^{(2)}+\frac{3}{2}\hat{\omega}^{(3)}+\frac{1}{2}\hat{\omega}^{(4)}+3\hat{\omega}^{(5)}+3\hat{\omega}^{(6)},\\
&\omega_2^{(12)}:=\hat{\omega}^{(1)}+2\hat{\omega}^{(2)}+3\hat{\omega}^{(3)}+6\hat{\omega}^{(4)}+6\hat{\omega}^{(5)}+\hat{\omega}^{(6)},\\
&\omega_3^{(12)}:=\frac{1}{2}\hat{\omega}^{(1)}+\frac{1}{4}\hat{\omega}^{(2)}+3\hat{\omega}^{(3)}+\hat{\omega}^{(4)}+\frac{3}{2}\hat{\omega}^{(5)}+\frac{3}{2}\hat{\omega}^{(6)}+\frac{1}{4}\hat{\omega}^{(7)},\\
&\omega_{4,1}^{(12)}:=\hat{\omega}^{(1)}+\hat{\omega}^{(2)},\\
&\omega_{4,2}^{(12)}:=6\hat{\omega}^{(1)}+6\hat{\omega}^{(2)}+6\hat{\omega}^{(3)}+6\hat{\omega}^{(4)}+2\hat{\omega}^{(5)}+3\hat{\omega}^{(6)}+3\hat{\omega}^{(7)},\\
&\omega_{5,1}^{(12)}:=\frac{1}{2}\hat{\omega}^{(1)}+\frac{1}{4}\hat{\omega}^{(2)}+\frac{1}{4}\hat{\omega}^{(3)}+3\hat{\omega}^{(4)}+\hat{\omega}^{(5)}+\frac{3}{2}\hat{\omega}^{(6)}+\frac{1}{4}\hat{\omega}^{(7)},\\
&\omega_{5,2}^{(12)}:=\frac{1}{4}\hat{\omega}^{(1)}+\frac{1}{4}\hat{\omega}^{(2)}+\frac{1}{2}\hat{\omega}^{(3)}+\frac{1}{4}\hat{\omega}^{(4)}\\
&\omega_{5,3}^{(12)}:=\frac{3}{2}\hat{\omega}^{(1)}+\frac{3}{2}\hat{\omega}^{(2)}+\frac{3}{2}\hat{\omega}^{(3)}+\frac{1}{2}\hat{\omega}^{(4)}+3\hat{\omega}^{(5)}+3\hat{\omega}^{(6)},\\
&\omega_6^{(12)}:=2\hat{\omega}^{(1)}+3\hat{\omega}^{(2)}+\frac{1}{3}\hat{\omega}^{(3)}+\frac{1}{3}\hat{\omega}^{(4)}+1\hat{\omega}^{(5)}+6\hat{\omega}^{(6)}+6\hat{\omega}^{(7)}+\frac{1}{3}\hat{\omega}^{(8)},
\end{align*}
where again we have written $\hat{\omega}^{(i)}=\bold{v}\otimes \cdots \otimes \left(\omega^{(i)}-\frac{c^{(i)}}{24}\mathbf{v}\right) \otimes \cdots \otimes \bold{v}$. As before we write $\hat{L}(0)$ to indicate the 0-mode of the vertex operators associated to the previous vectors. We also need modules for $e_4(\tau)$ and $o_4(\tau)$. It is possible to specify these modules implicitly by making use of equation \eqref{l43A}. In fact, using Corollary \ref{specALtrace} we can rewrite $\mu_{m,0}(z,\tau)$ in terms of characters of cone vertex algebras and 1-dimensional lattice vertex algebras. Furthermore, the theta functions $\theta_{m,r}$ also admits expressions in terms of trace functions of 1d lattice vertex algebras as described in section \ref{SVOAmodules}. It remains to find a module for the meromorphic Jacobi form
\begin{equation}
\label{Jac43A}
\psi^{(4)}_{3A}(z, \tau):=2i \theta_1(6z, 3\tau)\theta_1(z, \tau)^{-1}\theta_1(3z, 3\tau)^{-1}\eta(\tau)^3
\end{equation}
featuring in equation \eqref{eq:l=4}. Notice that constructing modules for these meromorphic functions is what is referred to as the ``meromorphic module problem" in \cite{supervertexmeromorphic}. It is easy to see that \eqref{Jac43A} also admits an expression in terms of characters of the modules discussed in \ref{SVOAmodules}. In fact we have, for $0<-{\rm Im}(z)<{\rm Im}(\tau)$,
\begin{equation}
\label{jacchar}
\psi^{(4)}_{3A}(z, \tau)=2i\chi^{A^{+}}(\tau)\chi^{A^{-}}(\tau)\chi^{A_{tw}}\left(6z+\frac{1}{2},3\tau\right)\chi^{\rA_{tw}}(z, \tau)\chi^{\rA_{tw}}(3z,3\tau).
\end{equation}
Using the relations $H^{(4)}_{3A,r}(\tau)=-H^{(4)}_{3A,-r}(\tau)$, and $\theta_{m,r}(z,\tau)=\theta_{m,-r}(z,\tau)$, we can give a prescription for the construction of modules\footnote{We can also express modules for $H^{(4)}$ implicitly in terms of vertex algebra modules by writing, for $z=a\tau+b$ with $a\in \mathbb{Q}^*$, $|a|<1$, $b \in \mathbb{R}$
\begin{equation}
\label{l4spec}
\begin{split}
&\sum\limits_{r=1}^3H^{(4)}_{3A,r}(\tau)\left[\theta_{4,r}(z,\tau)-\theta_{4,r}(-z,\tau)\right]
=-4e^{-16\pi i  b}q^{-2ma^2}\chi^{\mathcal{H}}(\tau)^2\tilde{T}^{(1)}_{\mathbf{a}, \mathbf{b}}(\tau)\\&+2i\chi^{A^{+}}(\tau)\chi^{A^{-}}(\tau)\chi^{A_{tw}}\left(6z+\frac{1}{2},3\tau\right)\chi^{\rA_{tw}}(z, \tau)\chi^{\rA_{tw}}(3z,3\tau)
-2\chi_{\frac{b}{16}}^{L^1+a}(\tau)\chi^{H}(\tau)
\end{split}
\end{equation}
where $\mathbf{a}=(1+a,0)$, $\mathbf{b}=(b,0)$ and we have written $\tilde{T}_{\mathbf{a}, \mathbf{b}}$ to indicate the cone vertex algebra trace function with quadratic form $\tilde{A}=\left(\begin{smallmatrix}8 & 1\\ 1 & 0 \end{smallmatrix}\right)$ in order to distinguish it from the trace functions with respect to $A=\left(\begin{smallmatrix} 1 & 1 \\ 1 & 0\end{smallmatrix}\right)$. Thus the McKay-Thompson series $H^{(4)}_{3A,r}$ are specified by the different $y$-powers in the right hand side of \eqref{l4spec}. Notice that the $z$ dependence influences, through $a$, which cone vertex algebra and one dimensional lattice modules will appear in the right hand side of \eqref{l4spec}} for $H^{(4)}_{3A,r}$ starting from equation \eqref{l43A}.
In fact, we can write $\Theta^{+}$ as
\begin{equation}
    \label{vectheta}
    \Theta^{+}_{A^{(4)}, \bold{c}^{(4)}_1, \bold{c}^{(4)}_2}(a\tau+b, 0;\tau)=2\sum\limits_{(n_1, n_2)\in \mathcal{C}}(-1)^{s(n_1,n_2)}y^{8n_1+n_2}q^{4n_1^2+n_1n_2}-\sum\limits_{n \in \mathbb{Z}}y^{8n}q^{4n^2}
\end{equation}
where $\mathcal{C}$ is the cone $\mathcal{C}:=\left\{(n_1,n_2)\in \mathbb{Z}^2:n_1\geq 0, n_2\geq 0\right\}\cup\left\{(n_1,n_2)\in \mathbb{Z}^2:n_1<0, n_2<0\right\}$ and $s$ corresponds to the sign automorphism
\begin{equation}
    s(n_1, n_2):=
    \begin{cases}
    1 & \text{if } n_1\geq0, n_2 \geq 0,\\
    -1 & \text{if } n_1<0, n_2<0.
    \end{cases}
\end{equation}
The vector space interpretation of the indefinite theta function \eqref{vectheta}, the vertex algebra interpretation of $\psi^{(4)}_{3A }$  \eqref{jacchar}, together with \eqref{l43A} 
give a definition of a bi-graded vector space $\mathcal{H}=\bigoplus\limits_{n,l}\mathcal{H}_{n,l}$ with an additional $\ZZ_2$-grading, that satisfies 
\begin{equation}
\sum\limits_{r=1}^3\frac{H^{(4)}_{3A,r}(\tau)}{2}\left[\theta_{4,r}(z,\tau)-\theta_{4,r}(-z,\tau)\right]=\sum_{n,l} \text{sdim}(\mathcal{H}_{n,l})q^ny^l  
\end{equation}
where {sdim} stands for the super dimension that takes the $\ZZ_2$ grading into account by including additional sign factors. We now define the operators $\tilde{L}_0$ and $\tilde{J}_0$ acting as $\tilde{L}_0 \mathbf{v}=n\mathbf{v}$, $\tilde{J}_0\mathbf{v}=l\mathbf{v}$ $\forall \mathbf{v}\in \mathcal{H}_{n,l}$. We can thus define a supertrace on $\mathcal{H}$ through
\begin{equation}
    \text{sTr}_{\mathcal{H}}q^{\tilde{L}_0}y^{\tilde{J}_0}:=\sum_{n,l} \text{sdim}(\mathcal{H}_{n,l})q^ny^l.
\end{equation}
Noticing that 
\begin{equation}
    \left[\theta_{4,r}(z,\tau)-\theta_{4,r}(-z,\tau)\right]=
    \sum\limits_{k\in \mathbb{Z}}\left(y^{4k+r}-y^{-(4k+r)}\right)q^{\frac{(4k+r)^2}{16}}
\end{equation}
we can specify $H^{(4)}_{3A,r}$ for $r=1,2,3$ with the previous notation through
\begin{equation}
\frac{H^{(4)}_{3A,r}(\tau)}{2}=\text{sTr}_{\tilde{\mathcal{H}}_r}q^{\tilde{L}_0-\left(\frac{\tilde{J}_0}{4}\right)^2}
\end{equation}
where 
\begin{equation}
\mathcal{H}_{r}=\bigoplus\limits_{n}\mathcal{H}_{n, l=r}.
\end{equation}
With this notation, we can rewrite the functions \eqref{e}, and \eqref{o} as
\begin{align}
&e_4(\tau)=-2\text{sTr}_{\tilde{\mathcal{H}}_2}q^{\tilde{L}_0-\left(\frac{\tilde{J}_0}{4}\right)^2},\\
&o_4\left(\frac{3}{2}\tau\right)=\text{sTr}_{\tilde{\mathcal{H}}_1}q^{\tilde{L}_0-\left(\frac{\tilde{J}_0}{4}\right)^2}-\text{sTr}_{\tilde{\mathcal{H}}_3}q^{\tilde{L}_0-\left(\frac{\tilde{J}_0}{4}\right)^2}.
\end{align}
We thus have 
\begin{thm}
	\label{l12th}
	The umbral McKay-Thompson series at lambency $\ell=12$ are specified by
	\begin{equation}
	\label{modules12}
	\begin{split}
	&\begin{split} 
	\left(H^{(12)}_{g,1}-H^{(12)}_{g,11}\right)(\tau)=
	{\rm tr}_{A}(g)\Biggl[&{\rm tr}_{\frak{M}^{(12)}_{1,1}}\left(g_{\left(\frac{1}{2},\frac{1}{2}\right)}q^{\hat{L}(0)}\right)+4 {\rm tr}_{\frak{M}^{(12)}_{1,2}}\left(g_{\left(\frac{1}{2},0\right)}q^{\hat{L}(0)}\right)\\
	&+{\rm tr}_{\frak{M}^{(12)}_{1,3}}\left(q^{\hat{L}(0)}\right)\Biggr]+{\rm tr}_{A}(g)o_4\left(\frac{\tau}{2}\right),
	\end{split}
	\\
	&\left(H^{(12)}_{g,3}-H^{(12)}_{g,9}\right)(\tau)=2{\rm tr}_{A}(g)tr_{\frak{M}^{(12)}_{3}}\left(g_{\left(\frac{1}{2},\frac{1}{2}\right)}q^{\hat{L}(0)}\right),\\
	&\begin{split}
	\left(H^{(12)}_{g,5}-H^{(12)}_{g,7}\right)(\tau)=
	{\rm tr}_{A}(g)\Biggl[&{\rm tr}_{\frak{M}^{(12)}_{5,1}}\left(g_{\left(\frac{1}{2},\frac{1}{2}\right)}q^{\hat{L}(0)}\right)+4{\rm tr}_{\frak{M}^{(12)}_{5,2}}\left(g_{\left(\frac{1}{2},0\right)}q^{\hat{L}(0)}\right)\\ &+{\rm tr}_{\frak{M}^{(12)}_{5,3}}\left(q^{\hat{L}(0)}\right)\Biggr]+{\rm tr}_{A}(g)o_4\left(\frac{\tau}{2}\right),
	\end{split}
	\\
	&H^{(12)}_{g,2}(\tau)=H^{(12)}_{g,10}(\tau)=2{\rm tr}_{B}(g){\rm tr}_{\frak{M}^{(12)}_{2}}\left(g_{\left(0,\frac{1}{2}\right)}q^{\hat{L}(0)}\right),\\
	&H^{(12)}_{g,4}(\tau)=H^{(12)}_{g,8}(\tau)=4{\rm tr}_{B}(g)\left[{\rm tr}_{\frak{M}^{(12)}_{4,1}}\left(g_{\left(0,\frac{1}{2}\right)}q^{\hat{L}(0)}\right)+2{\rm tr}_{\frak{M}^{(12)}_{4,2}}\left(q^{\hat{L}(0)}\right)\right],\\ 
	&H^{(12)}_{g, 6}(\tau)=4{\rm tr}_{B}(g){\rm tr}_{\frak{M}^{(12)}_{6}}\left(g_{\left(0,\frac{1}{2}\right)}q^{\hat{L}(0)}\right)+{\rm tr}_{B}(g)e_4\left(\frac{\tau}{3}\right),
	\end{split}
	\end{equation}
	where $g_{\bold{b}}$ acts as specified in \eqref{twisting} on the cone vertex algebra module in the tensor product and trivially on all the others. 
\end{thm}
\subsection{Lambency Sixteen}
The umbral group is $G=\mathbb{Z}/2\mathbb{Z}$. Using the same notation as before for the irreducible representations, we can write all the McKay-Thomposon series in terms of characters as
\begin{equation}
\begin{split}
&H^{(16)}_{g,2}\left(\tau-\frac{1}{2}\right)=H^{(16)}_{g,14}\left(\tau-\frac{1}{2}\right)=2{\rm tr}_{B}(g) \chi^{A^{+}}(4\tau)\chi^{A^{+}}(\tau)^2\chi^{H}(2\tau)\chi^{H}(8\tau)T^{(8)}_{\left(\frac{5}{8}, \frac{1}{8}\right), \left(\frac{1}{2},0\right)}(\tau),\\
&H^{(16)}_{g,4}(\tau)=H^{(16)}_{g,12}(\tau)=2{\rm tr}_{B}(g)q^{-\frac{1}{16}}\chi^{A^{+}}(\tau)\chi^{\rA}(8\tau, -\tau)T^{(8)}_{\left(\frac{3}{8}, \frac{1}{4}\right), \left(0, \frac{1}{2}\right)}(\tau),\\
&H^{(16)}_{g,6}\left(\tau-\frac{1}{2}\right)=H^{(16)}_{g,10}\left(\tau-\frac{1}{2}\right)=2{\rm tr}_{B}(g)\chi^{A^{+}}(4\tau)\chi^{A^{+}}(\tau)^2\chi^{H}(2\tau)\chi^{H}(8\tau)T^{(8)}_{\left(\frac{7}{8}, \frac{3}{8}\right), \left(\frac{1}{2},0\right)}(\tau),\\
&H^{(16)}_{g,8}(\tau)={\rm tr}_{B}(g)\left(2q^{-\frac{1}{16}}\chi^{A^{+}}(\tau)\chi^{\rA}(8\tau, -\tau)T^{(8)}_{\left(\frac{1}{8}, \frac{1}{2}\right), \left(0,\frac{1}{2}\right)}(\tau)+\chi^{A^{+}}(2\tau)^3\chi^{A^{-}}(4\tau)\chi^{H}(\tau)^2\chi^{H}(8\tau)\right),\\
&\sum\limits_{n=0,7}(-1)^n H^{(16)}_{g,2n+1}(8\tau)=2{\rm tr}_{A}(g)\chi^{A^{+}}(4\tau)\chi^{A^{+}}(\tau)^2\chi^{H}(8\tau)T^{(8)}_{\left(\frac{1}{4}, \frac{1}{4}\right), \left(0,0\right)}(\tau).
\end{split}
\end{equation}
Let's now consider the following tensor products of modules
\begin{align*}
&\frak{M}^{(16)}_1:={A^{+}_{tw}}^{\otimes 3}\otimes \mathcal{H} \otimes V^{(8)}_{\left(\frac{1}{4}, \frac{1}{4}\right)},\\
&\frak{M}^{(16)}_2:={A^{+}_{tw}}^{\otimes 3}\otimes\mathcal{H}^{\otimes 2}\otimes V^{(8)}_{\left(\frac{5}{8}, \frac{1}{8}\right)},\\
&\frak{M}^{(16)}_4:=A^{+}_{tw}\otimes\rA_{tw}\otimes V^{(8)}_{\left(\frac{3}{8}, \frac{1}{4}\right)}\\
&\frak{M}^{(16)}_6:={A^{+}_{tw}}^{\otimes 3}\otimes\mathcal{H}^{\otimes 2}\otimes V^{(8)}_{\left(\frac{7}{8}, \frac{3}{8}\right)},\\ 
&\frak{M}^{(16)}_{8,1}:=A^{+}_{tw}\otimes\rA_{tw}\otimes V^{(8)}_{\left(\frac{1}{8}, \frac{1}{2}\right)},\\
&\frak{M}^{(16)}_{8,2}:={A^{+}_{tw}}^{\otimes 3}\otimes A^{-}_{tw}\otimes \mathcal{H}^{\otimes 3},
\end{align*}
and the respective vectors 
\begin{align*}
&\omega_1^{(16)}:=\frac{1}{2}\hat{\omega}^{(1)}+\frac{1}{8}\hat{\omega}^{(2)}+\frac{1}{8}\hat{\omega}^{(3)}+\hat{\omega}^{(4)}+\frac{1}{8}\hat{\omega}^{(5)},\\
&\omega_2^{(16)}:=4\hat{\omega}^{(1)}+\hat{\omega}^{(2)}+\hat{\omega}^{(3)}+2\hat{\omega}^{(4)}+8\hat{\omega}^{(5)}+\hat{\omega}^{(6)},\\
&\omega_4^{(16)}:=\hat{\omega}^{(1)}+8\hat{\omega}^{(2)}+\hat{\omega}^{(3)},\\
&\omega_6^{(16)}:=4\hat{\omega}^{(1)}+\hat{\omega}^{(2)}+\hat{\omega}^{(3)}+8\hat{\omega}^{(4)}+\hat{\omega}^{(5)},\\
&\omega_{8,1}^{(16)}:=\hat{\omega}^{(1)}+8\hat{\omega}^{(2)}+\hat{\omega}^{(3)},\\
&\omega_{8,2}^{(16)}:=2\hat{\omega}^{(1)}+2\hat{\omega}^{(2)}+2\hat{\omega}^{(3)}+4\hat{\omega}^{(4)}+\hat{\omega}^{(5)}+\hat{\omega}^{(6)}+8\hat{\omega}^{(7)},\\
\end{align*}
using the same notation as before. Defining $\hat{L}(0)$ as usual we get
\begin{thm} 
\label{l16th}
The umbral McKay-Thompson series at lambency $\ell=16$ are specified by
\begin{equation}
\begin{split}
&H_{g,2}\left(\tau\right)=H_{g,14}\left(\tau\right)=2{\rm tr}_{B}(g){\rm tr}_{\frak{M}^{(16)}_2}\left(g_{\left(\frac{1}{2},0\right)}e^{\pi i \hat{L}(0)}q^{\hat{L}(0)}\right),\\
&H_{g,4}(\tau)=H_{g,12}(\tau)=2q^{-\frac{1}{16}}{\rm tr}_{B}(g){\rm tr}_{\frak{M}^{(16)}_4}\left(g_{\left(0,\frac{1}{2}\right)}q^{-J(0)}q^{\hat{L}(0)}\right),\\
&H_{g,6}\left(\tau\right)=H_{g,10}\left(\tau\right)=2{\rm tr}_{B}(g){\rm tr}_{\frak{M}^{(16)}_6}\left(g_{\left(\frac{1}{2},0\right)}e^{\pi i \hat{L}(0)}q^{\hat{L}(0)}\right),\\
&H_{g,8}(\tau)={\rm tr}_{B}(g)\left[2q^{-\frac{1}{16}}{\rm tr}_{\frak{M}^{(16)}_{8,1}}\left(g_{\left(0,\frac{1}{2}\right)}q^{-J(0)}q^{\hat{L}(0)}\right)+{\rm tr}_{\frak{M}^{(16)}_{8,2}}\left(q^{\hat{L}(0)}\right)\right],\\
&\sum \limits_{n=0,7}(-1)^n H_{g,2n+1}(\tau)=2{\rm tr}_{A}(g){\rm tr}_{\frak{M}^{(16)}_1}\left(g_{\left(0,0\right)}q^{\hat{L}(0)}\right),
\end{split}
\end{equation}
where $g_{\bold{b}}$ acts as specified in \eqref{twisting} on the cone vertex algebra module in the tensor product and trivially on all the others. Analogously, $J(0)$ acts non-trivially only on the Weyl module $\rA_{tw}$.
\end{thm}
\section{Conclusion and Outlook}
In this paper we showed how certain trace functions of cone vertex algebras are related to a certain family of indefinite theta functions of signature $(1,1)$. This family possesses interesting number theoretic properties and it is related to Appell-Lerch sums and Ramanujan's mock theta functions. For three instances of umbral moonshine, those with lambency $\ell=8,12,16$, this allowed us to construct modules for {the relevant finite groups}
%linear combinations of the McKay-Thompson series
in terms of cone vertex algebras and other known super vertex operator algebras modules. We end the paper with a collection of open questions and possible future directions. 
\begin{itemize}
\item We expect that the family of indefinite theta functions expressible as trace functions of cone vertex algebras can be extended by studying vertex algebras associated to cones with a more general shape than what used in \eqref{cone}. The condition in \eqref{ccondition} on the choice of $\mathbf{c}$ is  chosen to restrain the sum over the lattice vectors on the first and third quadrant of the plane. More general choices for the vectors $\mathbf{c}$ will lead to a  sum on different cones. 
\item
Another natural generalization is to investigate more general cone vertex algebras that can reproduce, through trace functions, indefinite theta functions of general signature $(r-n, n)$. 
In particular, it is worth investigating whether cone vertex algebras could be useful to gain a better understanding of the umbral moonshine phenomenon more generally, including the potential moonshine phenomenon involving all the optimal Jacobi theta functions classified in \cite{optimalmock}. 
As remarked in previous sections, all mock theta functions appearing in the McKay-Thompson series of umbral moonshine can be written in terms of the traces of cone vertex algebras discussed in this paper. 
The remaining challenge is thus to find expressions of the McKay-Thompson series that are compatible with the umbral group actions. While here we have limited our analysis to three instances of umbral moonshine with small umbral groups that turn out to act trivially on the cone vertex algerba structure, more involved group actions can certainly appear in other examples, akin to what happens in \cite{moduleE8}. Furthermore, we note that the trace functions of the cone vertex algebras seem to connect the McKay-Thompson series to the meromorphic Jacobi forms associated to various instances of umbral moonshine, as a consequence of Corollary \ref{specALtrace}.  
\item Finally, it would be interesting to investigate the physical significance of the cone vertex algebras. Vertex operator algebras provide  a mathematical axiomatization of the chiral algebra of conformal field theories in two dimension and it would be interesting to understand what kind of conformal fields theories cone vertex algebras are related to. 
For instance, 
it is known that the specialized Appell-Lerch sum \eqref{AL2} captures the non-modular part of the elliptic genus of non-compact supersymmetric coset models \cite{noncompactEG}, \cite{UM&K3}.
This could shed light on the still mysterious relation between umbral moonshine and string theories compactified on $K3$ surfaces \cite{UM&K3} (see also \cite{tasi} for more complete references). 
\end{itemize}

\paragraph{Acknowledgments:} We thank John Duncan for interesting comments on an earlier version of the draft. The work of M.C. and G.S. is supported by the NWO vidi grant (number
016.Vidi.189.182). The work of M.C. has also received support from ERC
starting grant H2020 \#640159.
\begin{appendices}
\section{Indefinite Theta Representations of Mock Theta Functions}
\label{ITrep}
For completeness, we include expressions for the mock theta functions used in this work in terms of indefinite theta functions. A more extensive list of expressions including all Ramanujan's mock theta functions can be found in \cite{book}. 
We have\footnote{Notice that in our notation $\theta_{1}(z,\tau)=\theta(-z,\tau)$ with $\theta(z,\tau)$ defined as in \cite{book}.}
\paragraph{Order 2}
\begin{align*}
A(q)&=\frac{q^2\eta(4\tau)}{2 \eta(2\tau)^2}\Theta^+_{\left(\begin{smallmatrix}
	1 &1 \\ 1 & 0
	\end{smallmatrix}\right),\left(\begin{smallmatrix}0 \\ 1\end{smallmatrix}\right),\left(\begin{smallmatrix}-1 \\ 1\end{smallmatrix}\right)}\left(3\tau,\tau+\frac{1}{2}, 4\tau\right), \\
B(q)&=\frac{q^{\frac{17}{8}}\eta(2\tau)}{2\eta(\tau)\eta(4\tau)}\Theta^+_{\left(\begin{smallmatrix}
	1 &1 \\ 1 & 0
	\end{smallmatrix}\right),\left(\begin{smallmatrix}0 \\ 1\end{smallmatrix}\right),\left(\begin{smallmatrix}-1 \\ 1\end{smallmatrix}\right)}\left(3\tau,2\tau+\frac{1}{2}, 4\tau\right). \\
\end{align*}
\paragraph{Order 3}
\begin{align*}
f(q)&=-2\frac{q^{\frac{25}{24}}}{\eta(\tau)}\Theta^+_{\left(\begin{smallmatrix}
	1 &1 \\ 1 & 0
	\end{smallmatrix}\right),\left(\begin{smallmatrix}0 \\ 1\end{smallmatrix}\right),\left(\begin{smallmatrix}-1 \\ 1\end{smallmatrix}\right)}\left(2\tau+\frac{1}{2},\frac{1}{2}\tau, 3\tau\right)+q^{\frac{1}{24}}\frac{\eta(3\tau)^4}{\eta(\tau)\eta(6\tau)^2},\\
\omega(q)&=\frac{q^{\frac{13}{12}}}{\eta(\tau)}\Theta^+_{\left(\begin{smallmatrix}
	1 &1 \\ 1 & 0
	\end{smallmatrix}\right),\left(\begin{smallmatrix}0 \\ 1\end{smallmatrix}\right),\left(\begin{smallmatrix}-1 \\ 1\end{smallmatrix}\right)}\left(3\tau,2\tau+\frac{1}{2}, 6\tau\right)+q^{-\frac{2}{3}}\frac{\eta(6 \tau)^4}{\eta(2\tau)\eta(3\tau)^2}.
\end{align*}
\paragraph{Order 6}
\begin{align*}
\sigma(q)&=q^{\frac{4}{3}}\frac{\eta(2\tau)\eta(3\tau)}{2\eta(\tau)\eta(6\tau)^2}\Theta^+_{\left(\begin{smallmatrix}
	1 &1 \\ 1 & 0
	\end{smallmatrix}\right),\left(\begin{smallmatrix}0 \\ 1\end{smallmatrix}\right),\left(\begin{smallmatrix}-1 \\ 1\end{smallmatrix}\right)}\left(3\tau,\tau+\frac{1}{2}, 6\tau\right), \\ 
\psi_6(q)&=q^{\frac{25}{24}}\frac{\eta(\tau)\eta(6\tau)}{2\eta(2\tau)\eta(3\tau)^2}\Theta^+_{\left(\begin{smallmatrix}
	1 &1 \\ 1 & 0
	\end{smallmatrix}\right),\left(\begin{smallmatrix}0 \\ 1\end{smallmatrix}\right),\left(\begin{smallmatrix}-1 \\ 1\end{smallmatrix}\right)}\left(\tau+\frac{1}{2},\frac{3}{2}\tau+\frac{1}{2}, 3\tau\right).
\end{align*}
\paragraph{Order 8}
\begin{align*}
T_0(q)&=\frac{q^{\frac{9}{4}}\eta(4\tau)}{2\eta(2\tau)\eta(8\tau)}\Theta^+_{\left(\begin{smallmatrix}
	1 &1 \\ 1 & 0
	\end{smallmatrix}\right),\left(\begin{smallmatrix}0 \\ 1\end{smallmatrix}\right),\left(\begin{smallmatrix}-1 \\ 1\end{smallmatrix}\right)}\left(5\tau+\frac{1}{2},\tau, 8\tau\right), \\
T_1(q)&=-\frac{q^{\frac{21}{4}}\eta(4\tau)}{2\eta(2\tau)\eta(8\tau)}\Theta^+_{\left(\begin{smallmatrix}
	1 &1 \\ 1 & 0
	\end{smallmatrix}\right),\left(\begin{smallmatrix}0 \\ 1\end{smallmatrix}\right),\left(\begin{smallmatrix}-1 \\ 1\end{smallmatrix}\right)}\left(7\tau+\frac{1}{2},3\tau, 8\tau\right), \\
U_0(q)&=\frac{q^{\frac{1}{2}}\eta(4 \tau)}{2\eta(8\tau)^2}\Theta^+_{\left(\begin{smallmatrix}
	1 &1 \\ 1 & 0
	\end{smallmatrix}\right),\left(\begin{smallmatrix}0 \\ 1\end{smallmatrix}\right),\left(\begin{smallmatrix}-1 \\ 1\end{smallmatrix}\right)}\left(\tau,\tau, 4\tau\right), \\
V_0(q)&=-\frac{iq^{\frac{1}{2}}}{\theta_1(-\tau,8\tau)}\Theta^+_{\left(\begin{smallmatrix}
	1 &1 \\ 1 & 0
	\end{smallmatrix}\right),\left(\begin{smallmatrix}0 \\ 1\end{smallmatrix}\right),\left(\begin{smallmatrix}-1 \\ 1\end{smallmatrix}\right)}\left(\tau,4\tau+\frac{1}{2}, 8\tau\right)-\frac{\eta(2\tau)^3\eta(4\tau)}{\eta(\tau)^2\eta(8\tau)}, \\
V_1(q)&=-\frac{iq^{\frac{3}{2}}}{2\theta_1(-\tau,8\tau)}\Theta^+_{\left(\begin{smallmatrix}
	1 &1 \\ 1 & 0
	\end{smallmatrix}\right),\left(\begin{smallmatrix}0 \\ 1\end{smallmatrix}\right),\left(\begin{smallmatrix}-1 \\ 1\end{smallmatrix}\right)}\left(3\tau,2\tau+\frac{1}{2}, 8\tau\right). \\
\end{align*}
To make contact with the notation used in section \ref{ITspecification}, we write the function $\Theta^+_{A,{\bf c},{\bf c'}}(\bold{z},\tau)$ in terms of indefinite thetas functions \eqref{indtheta} through relation \eqref{zabrel}.
We can thus rewrite
\paragraph{Order 2}
\begin{align*}
A(q)&=e^{-\frac{3 \pi i}{4}}q^{\frac{1}{8}}\frac{\eta(4\tau)}{2\eta(2\tau)^2}\Theta_{\left(\frac{3}{4}, \frac{1}{4}\right), \left(0,\frac{1}{2}\right)}(4\tau),\\
B(q)&=e^{-\frac{3 \pi i}{4}}q^{-\frac{1}{2}}\frac{\eta(2\tau)}{2\eta(\tau)\eta(4\tau)}\Theta_{\left(\frac{3}{4}, \frac{1}{2}\right), \left(0,\frac{1}{2}\right)}(4\tau).\\
\end{align*}

\paragraph{Order 3}
\begin{align*}
f(q)&=-2e^{-\frac{5\pi i}{6}}\frac{q^{\frac{1}{24}}}{\eta(\tau)}\Theta_{\left(\frac{2}{3},\frac{1}{6}\right),\left(\frac{1}{2},0\right)}(3\tau)+q^{\frac{1}{24}}\frac{\eta(3\tau)^4}{\eta(\tau)\eta(6\tau)^2}, \\
\omega(q)&=e^{-\frac{\pi i}{2}}\frac{q^{-\frac{2}{3}}}{\eta(2\tau)}\Theta_{\left(\frac{1}{2}, \frac{1}{3}\right), \left(0, \frac{1}{2}\right)}(6\tau)+q^{-\frac{2}{3}}\frac{\eta(6 \tau)^4}{\eta(2\tau)\eta(3\tau)^2}.
\end{align*}
\paragraph{Order 6}
\begin{align*}
\sigma(q)&=e^{-\frac{\pi i}{2}}q^{\frac{1}{12}}\frac{\eta(2\tau)\eta(3\tau)}{2\eta(\tau)\eta(6\tau)^2}\Theta_{\left(\frac{1}{2}, \frac{1}{6}\right), \left(0, \frac{1}{2}\right)}\left(6\tau\right), \\
\psi_6(q)&=e^{-\frac{7\pi i}{6}}q^{\frac{3}{8}}\frac{\eta(\tau)\eta(6\tau)}{2\eta(2\tau)\eta(3\tau)^2}\Theta_{\left(\frac{1}{3}, \frac{1}{2}\right), \left(\frac{1}{2}, \frac{1}{2}\right)}\left(3\tau\right). 
\end{align*}
\paragraph{Order 8}
\begin{align*}
T_0(q)&=e^{-\frac{3 \pi i}{4}}q^{\frac{1}{16}}\frac{\eta(4\tau)}{2\eta(2\tau)\eta(8\tau)}\Theta_{\left(\frac{5}{8}, \frac{1}{8}\right), \left(\frac{1}{2},0\right)}(8\tau),\\
T_1(q)&=-e^{-\frac{5 \pi i}{4}}q^{-\frac{7}{16}}\frac{\eta(4\tau)}{2\eta(2\tau)\eta(8\tau)}\Theta_{\left(\frac{7}{8}, \frac{3}{8}\right), \left(\frac{1}{2},0\right)}(8\tau),\\
U_0(q)&=q^{\frac{1}{8}}\frac{\eta(4\tau)}{2\eta(8\tau)^2}\Theta_{\left(\frac{1}{4}, \frac{1}{4}\right), (0,0)}(4\tau),\\
V_0(q)&=-ie^{-\frac{\pi i}{8}}\frac{q^{-\frac{1}{16}}}{\theta_1( -\tau,8\tau)}\Theta_{\left(\frac{1}{8}, \frac{1}{2}\right), \left(0,\frac{1}{2}\right)}(8\tau)-\frac{\eta(2\tau)^3\eta(4\tau)}{\eta(\tau)^2\eta(8\tau)}, \\
V_1(q)&=-ie^{-\frac{3\pi i}{8}}\frac{q^{\frac{3}{16}}}{2\theta_1(-\tau,8\tau)}\Theta_{\left(\frac{3}{8}, \frac{1}{4}\right), \left(0,\frac{1}{2}\right)}(8\tau). \\
\end{align*}
\end{appendices}
\nocite{*}
\printbibliography
\end{document}